\documentclass[10pt,a4paper,twoside]{amsart} 
\usepackage[english]{babel}
\usepackage[latin1]{inputenc}
\usepackage[pdftex]{graphicx}
\usepackage{hyperref}
\usepackage{amssymb}
\usepackage{adjustbox}
\usepackage{amsmath}
\usepackage{latexsym}
\usepackage{amsthm}
\usepackage{verbatim}
\usepackage{tikz-cd}
\usepackage{enumitem}
\usepackage{afterpage}
\usepackage{xcolor,colortbl,soul}   
\usepackage{tabularx}
\counterwithin{figure}{section}
\newcolumntype{P}[1]{>{\centering\arraybackslash}p{#1}}
\usepackage{chngpage}
\usepackage{geometry}
\def\vertexsize {1.2pt}   
\newcommand{\vertex}[2][1]{\fill (#2) circle [radius = #1 * \vertexsize];}
\newcommand{\divisor}[2][1]{\fill [red] (#2) circle [radius = 2* \vertexsize];}
\newcommand{\divisorBig}[2][1]{\fill [red] (#2) circle [radius = 2.5* \vertexsize];}
\newcommand{\subgr}[2][1]{\fill [blue] (#2) circle [radius = 2* \vertexsize];} 
\usepackage{xcolor,color,soul}

\definecolor{verde}{rgb}{0.01, 0.75, 0.24}

\theoremstyle{plain}                       
\newtheorem{theo}{Theorem}[section]
\newtheorem{prop}[theo]{Proposition}    
\newtheorem{cor}[theo]{Corollary}       
\newtheorem{lem}[theo]{Lemma} 
\theoremstyle{definition}               
\newtheorem{defin}{Definition}

\newtheorem{ex}{Example}   
\newtheorem{rk}{Remark}   
\theoremstyle{remark} 
\title{Tropical Trigonal Curves: the general case} 
\date{}
\author{Margarida Melo, Angelina Zheng}

\address{Department of Mathematics and Physics \\ University of Roma Tre \\ 00146 Rome, Italy}
\email{margarida.melo@uniroma3.it}
\address{Department of Mathematics\\ University of T\"ubingen \\ 72076 T\"ubingen, Germany}
\email{zheng@math.uni-tuebingen.de}
\begin{document}
\begin{abstract}
    
    This paper is a follow-up of a previous work in which we show that, for a $3$-edge connected tropical curve $\Gamma$,   the existence of a divisor of degree $3$ and Baker-Norine rank at least $1$ in $\Gamma$ is equivalent to the existence of a non-degenerate harmonic morphism of degree $3$ from a tropical modification of $\Gamma$ to a tropical rational curve.
    In this work, we extend this result to a tropical curve with lower edge connectivity which does not contain a cycle of (at least three) separating vertices  (a so-called  necklace).
\end{abstract}

\thanks{MM is supported by MIUR via the projects  PRIN2017SSNZAW (Advances in Moduli Theory and Birational Classification),  PRIN 2022L34E7W (Moduli spaces and birational geometry) and PRIN 2020KKWT53 (Curves, Ricci flat Varieties
and their Interactions), and is a member of the Centre for Mathematics of the University of
Coimbra -- UIDB/00324/2020, funded by the Portuguese Government through FCT/MCTES. 
AZ is supported by Alexander von Humboldt Foundation.
MM and AZ are members of the INDAM group GNSAGA}
\maketitle
\tableofcontents
\section{Introduction}

The study of algebraic curves of given gonality is a classical subject in algebraic geometry, motivated by the rich geometry of such curves. Indeed, the fact that a smooth algebraic curve admits a $g^1_d$, i.e. a divisor of degree $d$ and rank at least $1$, is equivalent to the existence of a non-degenerate morphism of degree $d$ to $\mathbb P^1$. The geometry of these covers, which is deeply connected to their beautiful topology, can then be used to understand the moduli space of curves themselves, using their stratification by gonality. 

As one can expect, the topology of covers, under degenerations, translates into quite interesting combinatorial properties, the description of which is a challenge. The moduli space of admissible covers, first introduced by Harris and Mumford in \cite{HM82}, gives a remarkable answer to such a problem.
However, as shown in \cite{HM82}, the geometry of nodal curves admitting a $g^1_d$ is connected to the existence of a morphism from a modification of the source curve to a rational curve with possibly several components. This fact makes this moduli space much more interesting from the combinatorial point of view.

Due to the combinatorial relevance of this moduli problem, one can expect that its tropical version might be particularly helpful and interesting itself. 
Indeed, the case of tropical hyperelliptic curves works just as for smooth curves: it was proved by Melody Chan in \cite{MC} that a tropical curve admits a divisor of degree $2$ and rank $1$ (a $g_2^1$),  if and only if it admits a non-degenerate harmonic
morphism of degree $2$ to a tropical curve of genus $0$ (a metric tree).
Things are much trickier though in the case of curves of higher gonality. 
It is easy to show (see \cite{MZ25} for a proof) that if a metric graph $\Gamma$ is  $d$-gonal, i.e., if it is endowed with a non-degenerate harmonic morphism to a metric tree, then $\Gamma$ is divisorially $d$-gonal, i.e., it admits a divisor of degree $d$ and rank at least $1$ (obtained by pulling back a point from the tree via the morphism). 
On the other hand, the opposite implication does not hold in general. Indeed, in \cite[Example 5.13]{ABBR} the authors presented a metric graph which is divisorially trigonal, but non-trigonal (see also Example \ref{ex:Luo} for more details).

\bigskip

This paper is a follow-up of a previous work in which we proved that a $3$-edge connected metric graph $\Gamma$ (with at least four vertices in its canonical loopless model) is divisorially trigonal if and only if it is trigonal.

More precisely, the main result we proved in \cite{MZ25} in the $3$-edge connected case is the following.

\begin{theo}{\cite[Theorem 1.1]{MZ25}}\label{th:main3conn}
    Let $\Gamma$ be a $3$-edge connected metric graph with canonical loopless model $(G_{-},l_{-})$. The following are equivalent.
    \begin{itemize}
        \item[A.] $|V(G_{-})|=2,3$ or $\Gamma$ is trigonal.
        \item[B.]$\Gamma$ is divisorially trigonal.
        \end{itemize}
\end{theo}

We now consider the most general case, namely that of a metric graph $\Gamma$ with no assumption on its edge connectivity, i.e. a metric graph with bridges or pairs of disconnecting edges.
The analog statement for hyperelliptic graphs with no assumptions on their edge connectivity is known to hold on metric graphs, see \cite{MC}. 
Luo's example discussed in \cite[Example 5.13]{ABBR} shows that the naive hope that the above statement generalizes to any metric graph cannot hold with no restriction on the connectivity of the curve. However, in the present paper, we explain that by excluding from the statement a certain class of graphs (which includes Luo's example), Theorem \ref{th:main3conn}  still generalizes to graphs with no restriction on the edge connectivity. 

We will see that the graphs for which the above statement might not hold are characterized by a cycle on at least three vertices, which are all separating vertices. We will call these graphs \emph{necklaces}, see Definition \ref{def:neck}. Our main Theorem is then the following.

\begin{theo}
    Let $\Gamma$ be a metric graph which is not a necklace with canonical loopless model $(G_{-},l_{-})$. The following are equivalent.
    \begin{itemize}
        \item[A.] $|V(G_{-})|=2,3$ or $\Gamma$ is trigonal.
        \item[B.]$\Gamma$ is divisorially trigonal.
        \end{itemize}
\end{theo}

We saw in the $3$-edge connected case that, given a divisor of degree $3$ and rank $1$, the tuples of points in the support of its effective linearly equivalent divisors are disjoint and identifying the points in each tuple determines uniquely the morphism with the desired properties. This, however, might no longer be true when we consider necklaces: any two edges on a cycle whose vertices are all separating vertices of the graph form a $2$-edge cut. In particular, any effective divisor of degree $3$ supported on such a cycle is linearly equivalent to one with support at any of its other points, and then clearly the strategy used in the $3$-edge connected case might not work again.
Such a problem, presented in detail in Section \ref{sc:2}, explains why necklaces, which include \cite[Example 5.13]{ABBR}, are such that the equivalence between divisorial trigonality and trigonality cannot hold.

On the other hand, if we exclude necklaces, then a construction similar to the one in the $3$-edge connected case can be applied to prove that the equivalence between divisorial trigonality and trigonality does generalize to any metric graph with no assumption on its edge connectivity. 

Moreover, in Subsection \ref{ssc:admissible_covers}, such two definitions will also be related to that of tropical admissible cover of degree $3$ of \cite{CMR}: if we consider a graph with no separating vertices and no multiple edges all three definitions are equivalent. 

The relations between trigonality, divisorial trigonality and being a tropical admissible cover of degree $3$ are useful to describe the geometry of the tropicalization of the locus of trigonal curves along with their closure inside the moduli space of stable curves or inside the moduli space of admissible covers. 
Indeed, as shown in \cite{ACP22}, in order to study the topology of the (link of the) moduli space of tropical curves, one can restrict to the locus of curves which have no bridges, cut vertices, loops, weights and multiple edges.

The proof of the main theorem will be carried out as follows.
In Section \ref{sc:hyperelliptic} we will focus on metric graphs which are (partially) hyperelliptic.
By \cite[Proposition A.4]{CKK}, the edge-connectivity provides a lower bound on the smallest positive degree of a divisor of rank (at least) 1 on the combinatorial graph defining the canonical model of $\Gamma$.
In particular, from \cite[Lemma 5.3]{BN}, $3$-edge connected metric graph cannot be hyperelliptic.
By dropping the assumption on $3$-edge connectivity, we then require to consider graphs which might be (partially) hyperelliptic.
In most of these situations, the hyperelliptic sub-curve admits a non-degenerate harmonic morphism of degree 3 to a tree only up to doing tropical modifications on the curve, which are prescribed by the (partial) hyperelliptic structure. 
Such a non-degenerate harmonic morphism of degree 3 to a tree extends to the whole curve in most cases, with some exceptions, which will be treated in Subsection \ref{ssc:hyp_necklace}.

Finally, in Section \ref{sc:non_hyp}, we will consider instead the non-hyperelliptic case and construct the non-degenerate harmonic morphism of degree 3 to a tree, when this is possible.
For divisorially trigonal graphs for which such a morphism does not exist, in Subsection \ref{ssc:trig_necklace} we will construct instead a morphism which is still non-degenerate, harmonic and of degree $3,$ but whose target space is not a metric tree.

\medskip

\noindent 

\medskip
\section{Divisorial trigonality and trigonality for metric graphs}\label{sc:2}

We start by recalling the definitions of harmonic morphism and gonality for metric graphs.
For other definitions and notation, we refer to  \cite[Section 2]{MZ25}.

\begin{defin}\label{def_metric}
Let $\varphi:\Gamma=(G,l)\to \Gamma'=(G',l')$ be a morphism of metric graphs. The index on an edge $e \in E(G)$ is defined as 
\begin{equation}\label{index}
\mu_{\varphi}(e)=\begin{cases}
\frac{l'(\varphi(e))}{l(e)}\in\mathbb{Z}, &\text{if }\varphi(e)\in E(G'),\\
0,&\text{if }\varphi(e)\in V(G').
\end{cases}\end{equation}
A morphism of metric graphs is \textbf{non-degenerate} if for any $x\in V(G)$ there exists $e\in E_x(G)$ such that $\mu_{\varphi}(e)>0.$ 
If for any $x\in V(G)$ the quantity
$
    \sum_{\substack{e\in E_{x}(G)\\ \varphi(e)=e'}}\mu_{\varphi}(e)
$
is constant for any $e'\in E(G')$, we say that the morphism $\varphi$ is \textbf{harmonic} and we set $m_{\varphi} (x)=\sum_{\substack{e\in E_{x}(G)\\ \varphi(e)=e'}}\mu_{\varphi}(e).$ In this case, we denote the \textbf{degree} of the morphism as $\operatorname{deg}(\varphi)=\sum_{\substack{e\in E(G)\\ \varphi(e)=e'}}\mu_{\varphi}(e)$ for some $e'\in E(G').$
\end{defin}

\begin{defin}\label{def:metric_gonal}
    A metric graph $\Gamma$ is \textbf{$d$-gonal} if there is a non-degenerate harmonic morphism of degree $d$ from a tropical modification $\Gamma'$ of $\Gamma,$ to a metric tree. 
\end{defin}

\begin{defin}\label{def: div_gonal}
    A metric graph is \textbf{divisorially $d$-gonal} if $W_d^1(\Gamma)=\{D \in {\rm Jac}{^d} (\Gamma): {\rm rk}(D)\geq 1\}\neq \emptyset.$ 
\end{defin}

We will consider metric graphs of genus $g> 2.$ Then, as observed in \cite[Remark 9]{MZ25}, the tropical versions of Riemann's and Clifford's theorems, proved respectively in \cite[Theorem 3.6]{AC} and \cite[Theorem A.1]{L}, hold. In particular, divisors of degree $3$ and rank at least 1 have rank precisely $1,$ i.e. $W_3^1(\Gamma)=W_3^1(\Gamma)\setminus W_3^2(\Gamma).$ 
We will also make extensive use of Dhar's burning algorithm \cite{D}, as explained in \cite[Remark 3]{MZ25}, as a sufficient condition to rule out the possibility that a divisor has positive rank, following \cite{BS}.

Let us consider first some examples, which show that the equivalence between trigonality and divisorial trigonality, proved in \cite[Theorem 1.1]{MZ25} for $3$-edge-connected metric graphs, does not hold in general.

To see this, we will use the following results.

\begin{lem}\label{lm:harm_cycle}
Let $G$ be a cycle, then there does not exist a non-degenerate degree 3 harmonic morphism $\varphi:\Gamma=(G,l)\to \Gamma_T=(T,l_T),$ where $T$ is a tree and such that
there are (at least) $3$ points $x_1, x_2,x_3 \in V(G)$ such that $$m_{\varphi}(x_i)=\sum_{\substack{e\in E_{x_i}(G)\\ \varphi(e)=e'}}\mu_{\varphi}(e)=3,$$
for any $i=1, 2, 3,$ $e'\in E_{\varphi(x_i)}(T).$
\end{lem}

\begin{proof}
    Assume by contradiction that such a morphism exists.
    Let us first observe that $T$ cannot have vertices of valence bigger than $2$, i.e. $T$ is a path. Indeed, given a vertex $v\in V(T)$, by harmonicity there must be a vertex in $\varphi^{-1}(v)$ of (at least) the same valence in $G$. Since $G$ is a cycle, all of its vertices have valence $2$, so the valence of $v$ is at most $2$.
    
    Let $t\in V(T)$ a vertex of valence $2$. 
    If $x\in V(G)$ is such that $\varphi(x)=t,$ then the two distinct edges $e_1,e_2\in E_x(G)$ are such that 
    $\varphi(e_1)=e_1^t,$ $\varphi(e_2)=e_2^t$ with $e_1^t,e_2^t$ two distinct edges incident to $t$ in $T$. Notice that neither $e_1,$ nor $e_2$ can be contracted: by harmonicity, if one is contracted the other edge should be contracted as well, but this would contradict the non-degeneracy of the morphism at $x$. 
    Moreover, again by harmonicity, we have
    $m_{\varphi}(x)=\mu_{\varphi}(e_1)=\mu_{\varphi}(e_2).$ Notice that both multiplicities cannot be $3:$ this would mean that there is no other edge in the pre-image of $e_j^t$ and in particular the removal of either $e_1$ or $e_2$ would disconnect $G$ which is not possible since $G$ is a cycle. 
    Then $x\neq x_i$, and the points $x_i$ cannot be sent to vertices in the tree of valence $2$, hence they have to be sent to the leaves of $t$. 
    
    Notice moreover that the $x_i$ have to be mapped to distinct leaves: if $\varphi(x_i)=\varphi(x_j)=t_l$ with $\operatorname{val}(t_l)=1$ then the morphism would have degree $m_{\varphi}(x_i)+m_{\varphi}(x_j)=6.$
    
    Then a morphism of degree $3$ cannot exist since there are only $2$ leaves in $T$.
\end{proof}

Let us recall that a \textbf{tropical modification} of a metric graph $\Gamma$ is a metric graph $\Gamma'$ obtained by gluing metric trees at some points of $\Gamma$.

\begin{lem}\label{lm:harm_cycle2} 
Let $G$ be a cycle and assume that there exists a non-degenerate degree 3 harmonic morphism $\varphi:\Gamma'\to \Gamma_T=(T,l_T),$ where $\Gamma'=(G',l')$ is a tropical modification of $\Gamma=(G,l)$ and $T$ is a tree.
Then the existence of (at least) $3$ points $x_1, x_2,x_3 \in \Gamma$ such that $$m_{\varphi}(x_i)=2,$$
for any $i=1,2,3,$
imposes conditions on $d(x_i,x_j)$ for $i,j\in\{1,2,3\}$.
\end{lem}

\begin{proof}
Notice that in general, given an harmonic morphism of degree $d$ from a metric graph to a tree, this can always be extended to a tropical modification of both graphs, by adding trees to the image and accordingly to the domain.
Therefore we will assume that any pre-image of an edge in $T,$ contains an edge in $G'$, contained in $\Gamma.$

We assume that such a morphism exists and show that the distances $d(x_i,x_j)$ cannot be arbitrary.
By definition of multiplicity, for each $i=1,2,3$, we have that there is an edge $e'\in E_{\varphi(x_i)}$ for which either there is an edge $e\in E_{x_i}(G')$ which is mapped by $\varphi$ to $e'$ with index $2$, or there are two distinct edges in $E_{x_i}(G')$ which are both mapped to $e'$ with index $1$.

Let us consider $x_1.$ we first observe that there exists a subset  $e_2 \subset \Gamma$ for which $x_i\in e_2$ and $\varphi$ have positive index on $e_2.$ 
Indeed, if not, this would mean that the two edges in $G'$, contained in $\Gamma$, incident to $x_i$ are contracted via the morphism and this contradicts harmonicity unless the whole cycle is contracted, but this would contradict our previous assumption on the fact that any pre-image of an edge in $T,$ contains an edge in $G'$, contained in $\Gamma$, which will have non-zero index.

Let us assume that $e_2$ is maximal with respect to inclusion.
We will see that $x_2,x_3$ will also be contained in $e_2$, which yields constraints in the distances $d(x_i,x_j)$.

As observed before, $\mu_{\varphi}(e_2)=2$ or $\mu_{\varphi}(e_2)=1$.
We will first consider the case $\mu_{\varphi}(e_2)=2$ and show that the morphism must be as in Figure \ref{fg:harm_cycle_mult}. 

    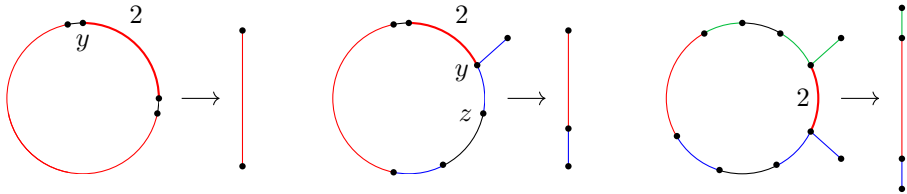
\begin{figure}[ht]\begin{tikzcd}
    \begin{tikzpicture}
    \begin{scope}
    \clip (-1,-1) rectangle (-0.2,0.98);
    \draw[red](0,0) circle (1);
    \end{scope}
    \begin{scope}
    \clip (-1,-1) rectangle (1,-0.2);
    \draw[red](0,0) circle (1);
    \end{scope}
    \begin{scope}
    \clip (0,0) rectangle (1.1,1.1);
    \draw[red,thick](0,0) circle (1);
    \end{scope}
    \begin{scope}
    \clip (-0.2,0.98) rectangle (0,1);
    \draw(0,0) circle (1);
    \end{scope}
    \begin{scope}
    \clip (0.98,-0.2) rectangle (1,0);
    \draw(0,0) circle (1);
    \end{scope}
    \vertex{1,0}
    \vertex{0,1}\vertex{-0.2,0.98}\vertex{0.98,-0.2}
    \draw[->] (1.3,0)--(1.8,0);
    \draw[red] (2.1,-0.9)--(2.1,0.9);
    \draw (0.7,1) node {$2$};\vertex{2.1,-0.9} 
    \draw (0,0.7) node {$y$};\vertex{2.1,0.9}
    \end{tikzpicture}
    &
    \begin{tikzpicture}
    \begin{scope}
    \clip (-1,-1) rectangle (-0.2,0.98);
    \draw[red](0,0) circle (1);
    \end{scope}
    \begin{scope}
    \clip (0,0.44) rectangle (1.1,1.1);
    \draw[red,thick](0,0) circle (1);
    \end{scope}
    \begin{scope}
    \clip (-0.2,-1) rectangle (0.45,-0.88);
    \draw[blue](0,0) circle (1);
    \end{scope}
    \begin{scope}
    \clip (0.9,-0.2) rectangle (1,0.44);
    \draw[blue](0,0) circle (1);
    \end{scope}
    \begin{scope}
    \clip (0.45,-0.88) rectangle (0.98,-0.2);
    \draw[](0,0) circle (1);
    \end{scope}
      \begin{scope}
    \clip (-0.2,0.98) rectangle (0,1);
    \draw(0,0) circle (1);
    \end{scope}
    \draw[blue](0.9,0.44)--(1.3,0.8);
    \vertex{0.9,0.44}\vertex{1.3,0.8}
    \vertex{0,1}\vertex{-0.2,0.98}\vertex{0.98,-0.2}\vertex{-0.2,-0.98}\vertex{0.45,-0.88}
    \draw[->] (1.3,0)--(1.8,0);
    \draw[red] (2.1,-0.4)--(2.1,0.9);\draw[blue] (2.1,-0.9)--(2.1,-0.4);
    \draw (0.7,0.3) node {$y$};
    \draw (0.75,-0.3) node {$z$};
    \draw (0.7,1) node {$2$};\vertex{2.1,-0.9}    \vertex{2.1,-0.4}\vertex{2.1,0.9}
    \end{tikzpicture}&
    \begin{tikzpicture}
    \begin{scope}
    \clip (0,0) rectangle (0.5,1);
    \draw[](0,0) circle (1);
    \end{scope}
    \begin{scope}
    \clip (0.5,0.44) rectangle (0.9,0.9);
    \draw[verde](0,0) circle (1);
    \end{scope}
    \begin{scope}
    \clip (-0.5,0.86) rectangle (0,1);
    \draw[verde](0,0) circle (1);
    \end{scope}
    \begin{scope}
    \clip (0.9,-0.44) rectangle (1.1,0.44);
    \draw[red,thick](0,0) circle (1);
    \end{scope}
    \begin{scope}
    \clip (-1.1,-0.5) rectangle (-0.5,0.86);
    \draw[red](0,0) circle (1);
    \end{scope}
    \begin{scope}
    \clip (-1.1,-1) rectangle (-0.3,-0.5);
    \draw[blue](0,0) circle (1);
    \end{scope}
    \begin{scope}
    \clip (0.45,-0.88) rectangle (0.9,-0.44);
    \draw[blue](0,0) circle (1);
    \end{scope}
    \begin{scope}
    \clip (-0.3,-1.1) rectangle (0.45,-0.88);
    \draw[](0,0) circle (1);
    \end{scope}
    \draw[verde](0.9,0.44)--(1.3,0.8);
    \draw[blue](0.9,-0.44)--(1.3,-0.8);
    \vertex{0.9,0.44}\vertex{1.3,0.8}
    \vertex{0.9,-0.44}\vertex{1.3,-0.8}
    \vertex{0,1}\vertex{0.5,0.86}
    \vertex{-0.5,0.86}\vertex{-0.86,-0.5}\vertex{0.45,-0.88}\vertex{-0.3,-0.95}
    \draw[->] (1.3,0)--(1.8,0);
    \draw[verde] (2.1,0.8)--(2.1,1.2);
    \draw[red] (2.1,-0.8)--(2.1,0.8);
    \draw[blue] (2.1,-0.8)--(2.1,-1.2);
    \draw (0.8,-0.1) node {$2$};\vertex{2.1,-1.2}\vertex{2.1,0.8} \vertex{2.1,-0.8}\vertex{2.1,1.2}
    \end{tikzpicture}
\end{tikzcd}\caption{Non-degenerate degree 3 morphisms from a cycle to a tree (up to tropical modifications). Black edges have arbitrary lengths (also zero).}\label{fg:harm_cycle_mult}
\end{figure}
    
Let $e_1$ be the other edge in $E_{x_1}(G')$, contained in $\Gamma,$ and let $y=\overline{e_1}\cap\overline{e_2}$
Depending on $\mu_{\varphi}(e_1)\leq 1$ (by the maximality of $e_2$) we have one of the above possibilities.
\begin{itemize}
    \item If $\mu_{\varphi}(e_1)=0$, then $\varphi(y)$ must be is also leaf: if not, then by harmonicity there is a tree with index $2$ (or two identical trees with index $1$) glued at $y,$ whose third pre-image is then contained in $\Gamma.$ Such tree then must consist of a single edge and its endpoints must be identified with $y$ and the leaf of the edge glued at $y$, which is not possible.
    Then the morphism is the one represented in the left figure in \ref{fg:harm_cycle_mult}.
    \item If $\mu_{\varphi}(e_1)=1$ and $\varphi(e_1)=\varphi(e_2),$ then $m_{\varphi}(y)=3$ and by the proof of Lemma \ref{lm:harm_cycle}, $\varphi(y)$ is a leaf. In this case the morphism is represented in the left figure in \ref{fg:harm_cycle_mult}, up to contraction of one of the black edges.
    \item If $\mu_{\varphi}(e_1)=1$ and $\varphi(e_1)\neq\varphi(e_2),$ then by harmonicity there is a leaf-edge incident to $y$ with same image as ${\varphi}(e_1)$ with index $1$, as represented in the figure in the center in \ref{fg:harm_cycle_mult}. 
    Let us assume $e_1$ again to be maximal in $\Gamma$ with an image in $\Gamma_T$ with $\mu_{\varphi}(e_1)=1.$ Let $e_3,$ the other consecutive edge to $e_1$, contained in $\Gamma$. 
    Let $z=\overline{e_1}\cap\overline{e_3}.$ We have, by harmonicity, that $\mu_{\varphi}(e_3)\in\{0,1\}.$ 
    If $\mu_{\varphi}(e_3)=0,$ then by harmonicity we require the image of $e_3$ to be a leaf, and the morphism is again  represented in the central figure in \ref{fg:harm_cycle_mult}.
    If instead $\mu_{\varphi}(e_3)=1,$ then ${\varphi}(e_3)={\varphi}(e_1)$, hence $\varphi(z)$ is again a leaf and the morphism is again represented by the same figure, up to contraction of the black edge incident to $z.$
The whole argument can be repeated for the other edge consecutive edge to $e_2,$ which would yield the morphism represented in the right figure in \ref{fg:harm_cycle_mult}.
    
\end{itemize}

The above description shows then that all the possible morphisms, up to addition of leaves over the endpoints of the black edges (with appropriate indices) and arbitrary contraction of black edges, are the ones represented in Figure \ref{fg:harm_cycle_mult}, and in particular that there is a unique maximal edge with index $2,$ which therefore must contain all points $x_i.$
Let $L$ be the total length of the cycle and assume $i,j$ to be such that $d_{ij}=d(x_i,x_j)$ is maximum among all $x_k$ such that $m_{\varphi}(x_k)=2.$
The above then shows that 
$d_{ij}\leq L/3.$
In other words, the position of the points $x_k$ along the cycle cannot be arbitrary.

Let us consider instead the case where $\mu_{\varphi}(e_2)=1$ and consider $e_1$, the other incident edge to $x_1,$ in $\Gamma.$
One can refer to the latter case in the discussion (by regarding $e_2,e_1$ as the edges $e_1,e_3$) to conclude that $\mu_{\varphi}(e_1)=1$. The same then holds for $x_2,x_3$ and between any two consecutive vertices $x_i,x_j$ there must be a vertex $p_{ij}$ which is sent to a vertex of the tree. By harmonicity and the fact that the degree is $3$, we require that on such point there has to be some edge which is sent to the same image of an edge $e_k$ incident to $x_k$ with $k\neq i,j$. This forces $p_{ij}$ to have a leaf-edge, whose image via the morphism will be the same as $e_k$. 
    However, such a morphism is well defined only if the indices of the morphism agree with the edge-length relations: let us denote the edges $e_1,e_1',\dots,e_{3},e_3'$ such that $e_i,e_i',$ $e_i',e_{i+1}$ and $e_3',e_1$ are consecutive.
    Then we require $l(e_i)=l(e_i')=l_i$ for some $l_i\in\mathbb R_{>0}$ (and moreover the length of the leaf-edge at any vertex $p_{ij}$ to be $l_k$ with $k\neq i,j$), as in Figure \ref{fg:harm_cycle}.

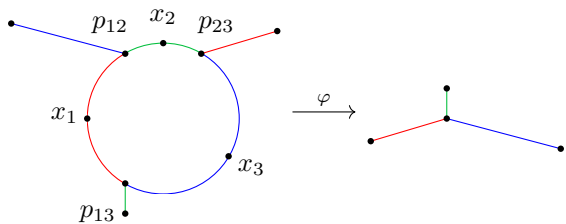
\begin{figure}[ht]\begin{tikzcd}
\begin{tikzpicture}
    \begin{scope}
    \clip (-1,-1) rectangle (-0.5,1.3);
    \draw[red](0,0) circle (1);
    \end{scope}
    \begin{scope}
    \clip (-0.5,0.86) rectangle (0.5,1.3);
    \draw[verde](0,0) circle (1);
    \end{scope}
    \begin{scope}
    \clip (-0.5,-1) rectangle (1,0.86);
    \draw[blue](0,0) circle (1);
    \end{scope}
    \draw[blue](-0.5,0.86)--(-2,1.26);\vertex{-2,1.26}
    \draw[red](0.5,0.86)--(1.5,1.16);\vertex{1.5,1.16}
    \draw[verde](-0.5,-0.86)--(-0.5,-1.26);\vertex{-0.5,-1.26}
    \draw (-1.3,0) node {$x_1$};
    \draw (0,1.3) node {$x_2$};
    \draw (1.16,-0.7) node {$x_3$};
    \draw (0.7,1.2) node {$p_{23}$};
    \draw (-0.7,1.2) node {$p_{12}$};
    \draw (-0.86,-1.3) node {$p_{13}$};
\vertex{-1,0}\vertex{0,1}\vertex{0.86,-0.5} \vertex{0.5,0.86}\vertex{-0.5,0.86}\vertex{-0.5,-0.86}
    \end{tikzpicture}\arrow[r,"\varphi"] &
    \begin{tikzpicture}
    \draw[verde] (0,0.4)--(0,0);  
    \draw[red] (-1,-0.3)--(0,0);  
    \draw[blue] (1.5,-0.4)--(0,0);  
    \vertex{0,0.4}\vertex{-1,-0.3}\vertex{1.5,-0.4}\vertex{0,0}
    \end{tikzpicture}
\end{tikzcd}\caption{A non degenerate degree 3 morphism from a cycle to a tree (up to tropical modifications).}\label{fg:harm_cycle}
\end{figure}
    
    Clearly, the above does not apply if we have more than $3$ vertices $x_i$ because otherwise the morphism would have a higher degree. If instead the vertices are precisely $3$ the existence of such a morphism gives constraints on the position of the points $x_i$. 
    More precisely, we require $d(x_i,x_j)<d(x_i,x_k)+d(x_j,x_k)$ for any $i,j,k\in \{1,2,3\}; i\neq j\neq k.$ 
    Then it is always possible to find a morphism as the one represented in Figure \ref{fg:harm_cycle}, where the position of the points $p_{ij}$ is determined by solving a linear system of $3$ equations defined by writing each distance $d(x_i,x_j)$ as the sum of $d(x_i,p_{ij})$ and $d(x_j,p_{ij})$, where $d(x_i,p_{ij})=d(x_i,p_{ik})$.
\end{proof}

\begin{ex}\label{ex:Luo}
The metric graph $\Gamma$ represented in Figure \ref{fg:Luo}, assuming that the blue edges have equal lengths, is divisorially trigonal. In fact, the degree $3$ divisor $D:=p_1+p_2+p_3\sim 3p_i,$ for any $i=1,2,3$ is easily seen to have rank $1$. However, as explained in \cite[Example 5.13]{ABBR}, there exists no tropical modification of $\Gamma$ admitting a non-degenerate harmonic morphism of degree $3$ to a tree. 

Notice that there exists no non-degenerate degree $2$ harmonic morphism from each of the components glued to the internal cycle to a tree (in blue in the picture).
Therefore, if such a morphism exists, then it has to be of degree $3$ when resticted to each of such components and it has to induce a non-degenerate harmonic morphism of degree $3$ from the subgraph defined by the $3$ edges of the internal cycle (in blue in the picture) to a tree.

Such a morphism cannot exist because it would require, by harmonicity, that the local multiplicity at any separating vertex is $3$ and this is not possible by Lemma \ref{lm:harm_cycle}. Notice that this also holds if we consider tropical modification: if such a morphism would exist after adding leaves, by harmonicity such leaves would have to correspond to some other edges of the graph, not in the internal cycle, thus increasing the degree.

\begin{figure}[h!]\begin{tikzcd}
\begin{tikzpicture}
    \draw(0.8,0.7)[] to [out=240, in=120] (0.8,-0.3);
    \draw(0.8,0.7)[] to [out=300, in=60] (0.8,-0.3);
    \draw[](0.8,-0.3)--(0.8,0.7);
    \vertex{0.8,0.7}
    \draw(-0.85,-2.2)[] to [out=75, in=195] (0,-1.5);
    \draw(-0.85,-2.2)[] to [out=15, in=255] (0,-1.5);
    \draw[](0,-1.5)--(-0.85,-2.2);\vertex{-0.85,-2.2}
    \draw(1.6,-1.5)[] to [out=285, in=165] (2.45,-2.2);
    \draw(1.6,-1.5)[] to [out=345, in=105] (2.45,-2.2);
    \draw[](1.6,-1.5)--(2.45,-2.2);\vertex{2.45,-2.2}  
    \vertex{0,-1.5}\vertex{0.8,-0.3}\vertex{1.6,-1.5}
    \draw[blue] (0,-1.5) to [out=340, in=200] (1.6,-1.5);
    \draw[blue] (0,-1.5) to [out=80, in=190] (0.8,-0.3);
    \draw[blue] (1.6,-1.5) to [out=110, in=350] (0.8,-0.3);
    \draw (0,-1.5) node[anchor=south west] {$p_2$};
    \draw (0.8,-0.3) node[anchor=north] {$p_1$};
    \draw (1.6,-1.5) node[anchor=south east] {$p_3$};
    \end{tikzpicture}    
\end{tikzcd}\caption{Luo's example: a divisorially trigonal metric graph $\Gamma$, which is not trigonal. The edges of $\Gamma$ forming the internal cycle (colored in blue), are assumed to  have the same length.}\label{fg:Luo}
\end{figure}
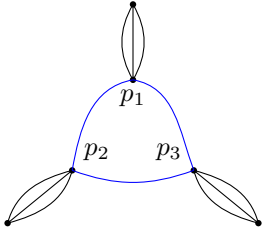

Instead, one can construct a non-degenerate degree $3$ harmonic morphism to a metric graph which contains the same cycle formed by the three edges of the same length, but with lengths multiplied by a factor of $3$, as in Figure \ref{fg:Luo_morph}.
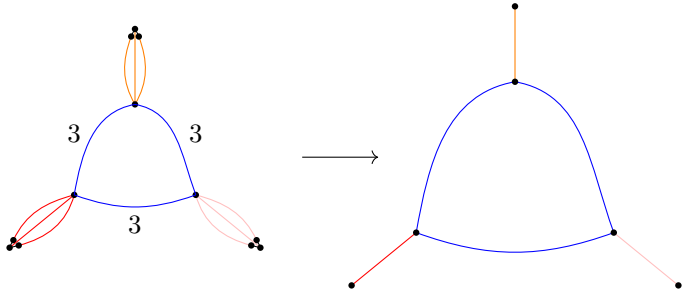
\begin{figure}[ht]\begin{tikzcd}
\begin{tikzpicture}
    \draw(0.8,0.7)[] to [out=240, in=70] (0.75,0.6);
    \draw(0.75,0.6)[orange] to [out=250, in=120] (0.8,-0.3);\vertex{0.75,0.6}
    \draw(0.8,0.7)[] to [out=300, in=110] (0.85,0.6);
    \draw(0.85,0.6)[orange] to [out=290, in=60] (0.8,-0.3);
    \draw[orange](0.8,-0.3)--(0.8,0.7);
    \vertex{0.8,0.7}\vertex{0.85,0.6}
    \draw(-0.85,-2.2)[] to [out=75, in=245] (-0.8,-2.1);
    \draw(-0.8,-2.1)[red] to [out=65, in=195] (0,-1.5);\vertex{-0.8,-2.1}
    \draw(-0.85,-2.2)[] to [out=15, in=195] (-0.73,-2.17);
    \draw(-0.73,-2.17)[red] to [out=15, in=255] (0,-1.5);\vertex{-0.73,-2.17}
    \draw[red](0,-1.5)--(-0.85,-2.2);\vertex{-0.85,-2.2}
    \draw(1.6,-1.5)[pink] to [out=345, in=115] (2.40,-2.1);
    \draw(2.40,-2.1)[] to [out=295, in=105] (2.45,-2.2);\vertex{2.40,-2.1}
    \draw(1.6,-1.5)[pink] to [out=285, in=165] (2.33,-2.17);
    \draw(2.33,-2.17)[] to [out=345, in=165] (2.45,-2.2);\vertex{2.33,-2.17}
    \draw[pink](1.6,-1.5)--(2.45,-2.2);\vertex{2.45,-2.2} 
    \vertex{0,-1.5}\vertex{0.8,-0.3}\vertex{1.6,-1.5}
    \draw[blue] (0,-1.5) to [out=340, in=200] (1.6,-1.5);
    \draw[blue] (0,-1.5) to [out=80, in=190] (0.8,-0.3);
    \draw[blue] (1.6,-1.5) to [out=110, in=350] (0.8,-0.3);
    \draw (0,-0.8) node {$3$};
    \draw (0.8,-2) node{$3$};
    \draw (1.6,-0.8) node{$3$};
    \draw[->](3,-1)--(4,-1);    
    \draw[orange](5.8,0)--(5.8,1);
    \vertex{5.8,1}
    \draw[red](4.5,-2)--(3.65,-2.7);
    \draw[pink](7.1,-2)--(7.95,-2.7);
    \draw[blue] (4.5,-2) to [out=340, in=200] (7.1,-2);
    \draw[blue] (4.5,-2) to [out=80, in=190] (5.8,0);
    \draw[blue] (7.1,-2) to [out=110, in=350] (5.8,0);
    \vertex{5.8,0}\vertex{4.5,-2}\vertex{3.65,-2.7}\vertex{7.1,-2}\vertex{7.95,-2.7}
    \end{tikzpicture}    
\end{tikzcd}\caption{A non-degenerate harmonic morphism of degree $3$.}\label{fg:Luo_morph}
\end{figure}
\end{ex}

\begin{defin}
    \label{def:t3}
    We say that $\Gamma_{T_{\Delta}}=(T_{\Delta},l_{T_{\Delta}})$ is a metric \textbf{tree of triangles} if $T_{\Delta}$ is a graph without loops whose only (minimal) cycles are triangles of edges $\{e_1,e_2,e_3\}$ with $l_{T_{\Delta}}(e_i)=l_{T_{\Delta}}(e_j)$ for any $i,j=1,2,3$ with no edge in common.
\end{defin}

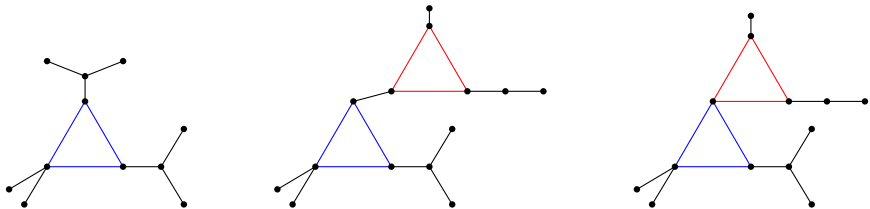
\begin{figure}[http]\begin{tikzcd}
\begin{tikzpicture}
    \draw(0,0)[blue]--(0.5,0.866);
    \draw(0,0)[blue]--(1,0);
    \draw(0.5,0.866)[blue]--(1,0);
    \draw(0.5,0.866)--(0.5,1.2);
    \draw(0,1.4)--(0.5,1.2);
    \draw(1,1.4)--(0.5,1.2);
    \draw(0,0)--(-0.5,-0.3);
    \draw(0,0)--(-0.3,-0.5);
    \draw(1,0)--(1.5,0);
    \draw(1.5,0)--(1.8,-0.5);
    \draw(1.5,0)--(1.8,0.5);
    \vertex{0,0}\vertex{0.5,0.866}\vertex{1,0}
    \vertex{0.5,1.2}\vertex{0,1.4}\vertex{1,1.4}
    \vertex{-0.5,-0.3}\vertex{-0.3,-0.5}\vertex{1.5,0}
    \vertex{1.8,-0.5}\vertex{1.8,0.5}
    \end{tikzpicture}   &
    \begin{tikzpicture}
    \draw(0,0)--(-0.5,-0.3);
    \draw(0,0)--(-0.3,-0.5);
    \draw(0,0)[blue]--(0.5,0.866);
    \draw(0,0)[blue]--(1,0);
    \draw(0.5,0.866)[blue]--(1,0);
    \draw(0.5,0.866)--(1,1);
    \draw(1,1)[red]--(1.5,1.866);
    \draw(1,1)[red]--(2,1);
    \draw(1.5,1.866)[red]--(2,1);\draw(1.5,0)--(1.8,-0.5);
    \draw(1.5,0)--(1.8,0.5);\draw(1,0)--(1.5,0);
    \vertex{0,0}\vertex{0.5,0.866}\vertex{1,0}
    \draw(2,1)--(2.5,1);\draw(3,1)--(2.5,1);\vertex{2.5,1}\vertex{3,1}
    \vertex{-0.5,-0.3}\vertex{-0.3,-0.5}\vertex{1.5,0}
    \vertex{1.8,-0.5}\vertex{1.8,0.5}
    \draw(1.5,1.866)--(1.5,2.1);
 \vertex{1,1}\vertex{1.5,1.866}\vertex{2,1}\vertex{1.5,2.1}
    \end{tikzpicture}   &
    \begin{tikzpicture}\draw(0,0)--(-0.5,-0.3);
    \draw(0,0)--(-0.3,-0.5);
    \draw(0,0)[blue]--(0.5,0.866);
    \draw(0,0)[blue]--(1,0);
    \draw(0.5,0.866)[blue]--(1,0);
    \draw(0.5,0.866)[red]--(1,1.732);
    \draw(0.5,0.866)[red]--(1.5,0.866);
    \draw(1,1.732)[red]--(1.5,0.866);
    \draw(1,1.732)--(1,2);
    \draw(1.5,0)--(1.8,-0.5);
    \draw(1.5,0)--(1.8,0.5);\draw(1,0)--(1.5,0);
    \draw(1.5,0.866)--(2,0.866);
    \draw(2,0.866)--(2.5,0.866);
    \vertex{0,0}\vertex{0.5,0.866}\vertex{1,0}
  \vertex{1,1.732}\vertex{1,2}\vertex{1.5,0.866}\vertex{2.5,0.866}\vertex{2,0.866}
    \vertex{-0.5,-0.3}\vertex{-0.3,-0.5}\vertex{1.5,0}
    \vertex{1.8,-0.5}\vertex{1.8,0.5}
    \end{tikzpicture}   
\end{tikzcd}\caption{Some examples of metric tree of triangle. Edges of the same color have same length.}\label{fg:triangle}
\end{figure}

However, there exist divisorially trigonal graphs, with no non-degenerate harmonic morphism of degree $3$ to a tree of triangles, or even a tree of $n$-cycles, of edges of the same lengths, with $n>3$. 

Let us recall the definition of separating vertex of a graph.
\begin{defin} A vertex $v\in V(G)$ is said to be \emph{separating} if $G$ admits a decomposition into two connected subgraphs $G_1$ and $G_2$ which have only $v$ in common. The graph $G$ can then be obtained as a disjoint union of $G_1$ and $G_2$, identified along the vertex $v$: we say that such a decomposition yields a separation of $G$.

\end{defin}
In particular, if $G$ has a loop incident at $v$, $v$ is always separating. Moreover, if there are no loops of $G$ incident to $v$, we have that $v$ is separating if and only if it is a cut vertex of $G$. 

\begin{defin}\label{def:neck}
    A \emph{necklace graph} is a connected graph $G$ that admits a decomposition into subgraphs $G_0, G_1, \dots, G_n$\sloppy, for $n\geq 3$, such that $G_0$ is a cycle, $G_i\cap G_j=\emptyset $ and $G_0\cap G_i=\{v_i\}$ is a separating vertex of $G$, $i=1\dots, n$.
    A \emph{necklace} is a metric graph whose underlying graph is a necklace graph.
\end{defin}

\begin{ex}\label{ex:neck}
We consider a metric graph $\Gamma$ that is a necklace, with a cycle $\gamma$ with separating vertices $x_1,\dots,x_4$ such that $\Gamma\setminus \gamma$ is given by components $\Gamma_i,$ with $\Gamma_i$ glued at $x_i$ and such that $2x_i\in W_2^1(\Gamma_i)$ for any $i=1,\dots,4,$ as in Figure \ref{fg:neck_hyp}.
    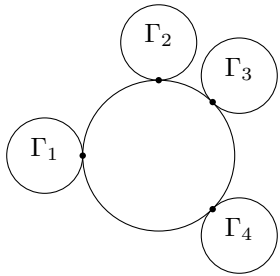
\begin{figure}[ht]\begin{tikzcd}
\begin{tikzpicture}
    \draw(0,0) circle (1);
    \draw(-1.5,0) circle (0.5);
    \draw (-1.5,0) node {$\Gamma_1$};
    \draw(0,1.5) circle (0.5);
    \draw (0,1.5) node {$\Gamma_2$};
    \draw(1.06,1.06) circle (0.5);
    \draw (1.06,1.06) node {$\Gamma_3$};
    \draw(1.06,-1.06) circle (0.5);
    \draw (1.06,-1.06) node {$\Gamma_4$};
    \vertex{-1,0}\vertex{0,1}\vertex{0.71,0.71}\vertex{0.71,-0.71}
    \end{tikzpicture}    
\end{tikzcd}\caption{A necklace of hyperelliptic graphs.}\label{fg:neck_hyp}
\end{figure}

The metric graph represented in figure \ref{fg:neck_hyp} is divisorially trigonal. Indeed, for any $3$ points $p_1,p_2,p_3$ in the cycle, define $D=p_1+p_2+p_3\in W_3^1(\Gamma).$

For any $w\in \gamma$, $D\sim 2w+w'$ for some $w'\in\gamma.$ To see this, it suffices to notice that any two edges in $\gamma$ form a $2$-edge cut, then we can define a rational function of slope $-1$ from the two closest points to $w$ between $p_1,p_2,p_3,$ each along a path of length $\operatorname{min}\{d(w,p_i);i=1,2,3\}$ towards $w$, and constant everywhere else.
This yields a linear equivalent divisor $w+p_1'+p_2'.$ Repeating again from the points $p_1',p_2'$ along paths towards $w$ yields the divisor $2w+w'.$ 
By the above argument, $D\sim 2x_i+x_i'$ for some $x_i'\in\gamma,$ hence $\Gamma$ is divisorially trigonal since $2x_i\in W_2^1(\Gamma_i)$. 

The metric graph is then divisorially trigonal for any number of edges in the cycle and for arbitrary lengths.
However, by Lemma \ref{lm:harm_cycle2}, the existence of a tropical modification admitting a non-degenerate harmonic morphism of degree $3$ to a metric tree depends on the position of the points $x_i$, therefore it might not exist.
\end{ex}

\begin{rk}\label{rk:neck_not_hyp}
    A $2$-edge connected necklace $\Gamma$ cannot be hyperelliptic.
    Assume indeed by contradiction there is $H\in W_2^1(\Gamma),$ and let $x$ be a separating vertex. By $2$-edge connectivity we have $\operatorname{val}(x)\geq 4,$ with two incoming edges contained in the cycle defining the necklace. Let $y\in \Gamma$ such that $H\sim x+y.$ One can check that starting the fire at either the component glued at $x$ or from the cycle burns the entire graph unless $\operatorname{val}(x)= 4$ and $y=x,$ hence $H\sim 2x.$ Then the rational function with slope $-1$ from $x$ along the two edges within the cycle for two paths of equal length, until the first one reaches the closest separating vertex, yields $2x\sim x_1+x_2,$ with $x_i$ a separating vertex for at least one $i\in\{1,2\}$ and $x_1\neq x_2.$ Starting Dhar's burning algorithm from the component glued at $x_i$ burns the whole graph, hence $x_1+x_2\sim H$ cannot have rank $1.$
\end{rk}

Notice that the metric graphs in Examples \ref{ex:Luo} and \ref{ex:neck} are necklaces.
If we do exclude such graphs, then \cite[Theorem 1.1]{MZ25} can be generalized to a metric graph with no assumption on its edge connectivity.

\begin{theo}\label{th:main_general}
    Let $\Gamma$ be a metric graph which is not a necklace with canonical loopless model $(G_{-},l_{-})$. The following are equivalent.
    \begin{itemize}
        \item[A.] $|V(G_{-})|=2,3$ or $\Gamma$ is trigonal.
        \item[B.]$\Gamma$ is divisorially trigonal.
        \end{itemize}
\end{theo}
\begin{proof}[Proof of $A.\Rightarrow B$]
    The proof of the implication $A.\Rightarrow B.$ follows precisely \cite[Theorems 3.2 and 3.9]{MZ25}.
\end{proof}

The proof of the main theorem in the other direction is more difficult. The goal is to construct a morphism from the datum of a divisor $D\in W_3^1(\Gamma).$ As in the $3$-edge connected case, this will be done by studying the combinatorics of the metric graph $\Gamma,$ which is much more complicated when bridges or pairs of disconnecting edges are allowed. 

Some metric graphs might indeed be partially hyperelliptic, and in this case, we will see that a non-degenerate harmonic degree $3$ morphism over the hyperelliptic part to a metric tree can be constructed locally, as prescribed by the hyperelliptic structure. Such a morphism, however, might not extend to the whole graph, as discussed in Example \ref{ex:neck}.

Instead, if we consider metric graphs which are necklaces, as in Example \ref{ex:Luo}, the non-degenerate harmonic degree $3$ morphism to a tree (even if we allow tropical modifications) does not exist in general. However, in some cases, we can still characterize divisorial trigonality as the existence of a non-degenerate harmonic degree $3$ morphism to a triangle of trees, for a certain class of graphs. 

\begin{defin}\label{def:s_trig_neck}
    A \emph{non-hyperelliptic necklace} is a necklace $\Gamma$ such that none of the connected components of $\Gamma\setminus \gamma$
    is hyperelliptic, with $\gamma$ as in Definition \ref{def:neck}.
\end{defin}

\begin{theo}\label{th:main_neck}
    Let $\Gamma$ be a \emph{non-hyperelliptic necklace}. The following are equivalent.
    \begin{itemize}
        \item[A.] There exist a non-degenerate harmonic morphism of degree $3$ $\varphi':\Gamma'\to \Gamma_{T_{\Delta}}$ with $\Gamma'$ a tropical modification of $\Gamma,$ and $\Gamma_{T_{\Delta}}$ a metric tree of triangles, as in Definition \ref{def:t3},
        such that the pre-image of any cycle in $\Gamma_{T_{\Delta}}$ is given by the same cycle whose edge lengths are divided by $3.$
        \item[B.]$\Gamma$ is divisorially trigonal.
        \end{itemize}
\end{theo}

Similarly to the main Theorem \ref{th:main_general}, we prove first that the existence of the morphism  determines a divisor of degree $3$ and rank at least $1$. The proof in the other direction will be discussed in Section \ref{ssc:trig_necklace}.

\begin{proof}[Proof of Theorem \ref{th:main_neck}, $A.\Rightarrow B$]
  Let us assume that there exists a non-degenerate harmonic morphism of degree $3$ to a metric tree of triangles $\varphi':\Gamma'\to \Gamma_{T_{\Delta}}$ with $\Gamma'$ a tropical modification of $\Gamma.$ 
    For simplicity, we also assume that there is precisely one cycle $\Gamma_{T_{\Delta}}$. In case of more cycles, the following argument can be repeated.
    Let us observe that from the definition of metric tree of triangles, the graphs obtained by removing all cycles is a disjoint union of metric trees.
    Let $t_i\in \Gamma_{T_{\Delta}}$ be a vertex of the cycle over which a tree $T_i$ is glued. Then, since 
    the pre-image of any cycle in $\Gamma_{T_{\Delta}}$ is given by the same cycle whose edges-lengths are divided by $3$ we have $D:=(\varphi')^*(t_i)=3p_i$, with $p_i$ the corresponding separating vertex in the cycle in $\Gamma$. 
    
    Moreover, since the edges of the cycle have the same length then one has $D:=(\varphi')^*(t)=3p_i\sim p_1+p_2+p_3$ which proves that $D-w\sim E$ for any point $w$ on the cycle.

    If instead $w\in (\varphi')^{-1}(T_i)$ one shows that $3p_i-w\sim E$ for some effective divisor $E$ using the same argument in $A.\Rightarrow B.$ in \cite[Theorems 3.2 and 3.9]{MZ25}. Finally, since the fact that edges on the cycle have same length also implies that $3p_i\sim 3p_j$ for $i\neq j$, we can repeat the argument $w\in (\varphi')^{-1}(T_j),$ with $j\neq i.$ Therefore the rank of $D$ must be $1.$
\end{proof}

Finally, we will also prove that the metric graphs considered in Theorem \ref{th:main_neck} do not exist for low genera: if the genus is assumed to be sufficiently small, then divisorial trigonality and trigonality coincide.

\begin{cor}\label{cor:main_g6}
    Let $\Gamma$ be a metric graph with canonical loopless model $(G_{-},l_{-})$ of genus $g\leq 5$. The following are equivalent.
    \begin{itemize}
        \item[A.] $|V(G_{-})|=2,3$ or $\Gamma$ is trigonal.
        \item[B.]$\Gamma$ is divisorially trigonal.
        \end{itemize}
\end{cor}

\subsection{Trigonality and tropical admissible covers of degree 3}\label{ssc:admissible_covers}
As proved in \cite[Proposition 4.7]{MZ25}, a (divisorially) trigonal $3$-edge connected metric graph admits a tropical admissible cover of degree $3$ of a metric tree, as defined in \cite{CMR}.

Let us recall that a \emph{tropical admissible cover} of tropical curves  $\psi:\Gamma_{\operatorname{src}}=(G_{\operatorname{src}},l_{\operatorname{src}})\to \Gamma_{\operatorname{tgt}}=(G_{\operatorname{tgt}},l_{\operatorname{tgt}})$ is a harmonic morphism satisfying the \emph{local Riemann-Hurwitz equation} at any point, i.e. such that for any $v\in V(G_{\operatorname{src}})$ if $v'=\psi(v),$
\begin{equation}\label{eq:RH}
    2=2m_{\psi}(v)-\sum_{e\in E_v(G_{\operatorname{src}})}(\mu_\psi(e)-1).
\end{equation}
Notice that we are only interested in metric graphs with no weights for our purposes.

First, let us observe that it is sufficient to consider the Riemann-Hurwitz inequality, as defined in \cite[Definition 6 (D)]{LC}

\begin{equation}\label{eq:RH2}
    2\leq 2m_{\psi}(v)-\sum_{e\in E_v(G_{\operatorname{src}})}(\mu_\psi(e)-1).
\end{equation}

Indeed, whenever the inequality is strict, the equality can be reached by gluing leaf-edges at $v$ with suitable indices.

Clearly, by definition, a metric graph that admits a tropical admissible cover of degree $3$ of a metric tree is trigonal, but the opposite implication, in general, might not hold if we drop the assumption on the edge connectivity. 

However, we show that if $\Gamma$ admits a degree $3$ harmonic morphism to a metric tree which is not a tropical admissible cover, then $\Gamma$ contains multiple edges or a separating vertex.

\begin{prop}\label{prp:adm_cover}
    Let $\Gamma$ be a trigonal metric graph of genus $g\geq 3,$ with a harmonic morphism (with no contractions\footnote{From \cite[Proposition 2.1]{MZ25}, given a harmonic morphism which doesn't contract loops, there is always a tropical modification of it with no contraction.}) $\varphi:\Gamma'\to\Gamma_T$ of degree $3$ from a tropical modification $\Gamma'$ of $\Gamma$ to a metric tree $\Gamma_T$. If $\varphi$ is not a tropical admissible cover, then $\Gamma$ contains multiple edges or a separating vertex.
\end{prop}
\begin{proof}
    Let us first notice that if $m_{\varphi}(v)=1,$ then the Riemann-Hurwitz inequality \eqref{eq:RH2} is always satisfied, which is precisely what happens in the $3$-edge connected case.
    Let us also remark that for any $v\in V(G_{\operatorname{src}})$ such that $\operatorname{val}(v)\leq 2$ we have $m_{\psi}(v)=\mu_\psi(e)$ for any $e\in E_v(G_{\operatorname{src}})$ and therefore \eqref{eq:RH2} holds.
    Therefore, there must exist $v\in V(G_{\operatorname{src}})$ with $\operatorname{val}(v)\geq 3$ such that $m_{\varphi}(v)>1$ and
    \begin{equation}\label{eq:no_RH}
    2> 2m_{\varphi}(v)-\sum_{e\in E_v(G_{\operatorname{src}})}(\mu_\varphi(e)-1).
\end{equation}
Since the total degree of the morphism is $3$, we have two possibilities.
\begin{itemize}
    \item If $m_{\varphi}(v)=2$, then \eqref{eq:no_RH} holds if there are (at least) two edges $e_1,e_2\in E_v(G_{\rm src})$ such that $\mu_{\varphi}(e_1)=\mu_{\varphi}(e_2)=2.$
    Let $e_3$ be a third non-leaf edge incident to $v$. By harmonicity, the index over $e_3$ is again 2 (or it is $1,$ but then there will be a fourth edge $e_4\in E_v(G_{\rm src})$ with $\varphi(e_3)=\varphi(e_4)$). We consider the first case since the second one can be treated analogously. 
    Assume by contradiction that $v$ is not a separating vertex. This means that the three edges $e'_1,e'_2,e'_3$ with common vertex $w$, such that $\varphi(e_i')=\varphi(e_i)$ and $\mu_{\varphi}(e_i')=1,$ are edges of $\Gamma$: there exist a connected subcurve $\Gamma_i$ connecting $e_i'$ and the endpoint of $e_i$, distinct from $v.$ Since $g\geq 3$ we may also assume $\Gamma_i$ has positive genus for some $i$.
    
    From Theorem \ref{th:main_general} $A.\Rightarrow B.$, the harmonic morphism defines a divisor $2v+w\in W_3^1(\Gamma).$ Depending on the edge lengths of $e_i,e_i'$ we have a linear equivalence $2v+w\sim 2v_i+y_i$ with $v_i\neq v$ the other endpoint of $e_i$ and $y_i\in e_i$, or $2v+w\sim v_i+w_i+z_i$ with $w_i$ the endpoint of $e_i'$ distinct from $w$ and $z_i\in e_i.$ 
    
    We start Dhar's burning algorithm from any $\Gamma_i$ one can check that the fire doesn't burn entirely the graph only if each $\Gamma_i$ is defined by multiple edges, as in Figure \ref{fg:g5}. In fact, if $\Gamma_i$ is not defined by multiple edges but contains any additional vertex $z$ of valence at least $3$, one could consider the linear equivalent divisor with support in $\Gamma_i\cap \overline{\Gamma\setminus\Gamma_i}$ and starting Dhar's burning algorithm from $z$ would burn the whole graph.

    \begin{figure}[ht]
\begin{tikzcd}
\begin{tikzpicture}
        \draw[thick,verde] (0,1)--(1,1);
        \draw[verde] (-1,0)--(1,0);
        \draw[] (-1,0) to [out=10, in=270] (0,1);
        \draw[] (-1,0) to [out=80, in=200] (0,1);

        \draw[thick, red] (1,1)--(1.5,1.75);
        \draw[red] (1,0)--(2.5,2);
        \draw[] (1.5,1.75) to [out=0, in=220] (2.5,2);
        \draw[] (1.5,1.75) to [out=30, in=180] (2.5,2);
        
        \draw[dashed,thick,blue](1,1)--(1.7,1);
        \draw[thick,blue] (1.7,1)--(2,1);
        \draw[blue] (1,0)--(3,0);
        \draw[] (2,1) to [out=280, in=170] (3,0);
        \draw[] (2,1) to [out=350, in=100] (3,0);
        \draw[] (2,1)--(3,0);
        \draw (1,1) node[anchor=north] {$v$};
        \draw (1,0) node[anchor=north] {$w$};
        \draw (0.5,1) node[anchor=south] {$2$};
        \draw (1.5,1) node[anchor=south] {$2$};
        \draw (1.2,1.5) node[anchor=east] {$2$};
        \vertex{1,1}
        \vertex{0,1}
        \vertex{2,1}
        \vertex{3,0}
        \vertex{-1,0} \vertex{1,0}\vertex{1.5,1.75}\vertex{2.5,2}
        \draw[->](3.5,0.5) to (4.5,0.5);
        \draw(4,0.5) node[anchor=south]{$\varphi$};
        \draw[verde] (5,0)--(8,0);
        \draw[red] (7,0)--(8.5,2);
        \draw[blue] (7,0)--(9,0);
        \vertex{9,0}
        \vertex{5,0} \vertex{7,0}\vertex{8.5,2}
    \end{tikzpicture}\end{tikzcd}\caption{A harmonic degree $3$ morphism from a graph of genus $5$ to a tree, which doesn't satisfy Riemann-Hurwitz inequality at $v$.}\label{fg:g5}
\end{figure}
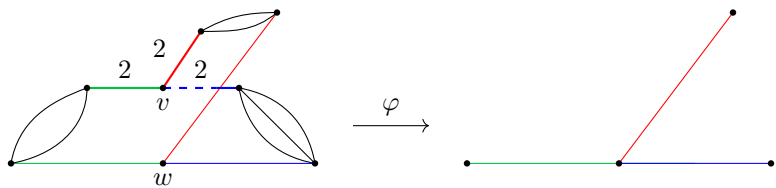
    
    \item If instead $m_{\varphi}(v)=3$ then this means that $v$ is the unique vertex in the preimage of a vertex in the tree, hence $v$ must be a separating vertex.
\end{itemize}
\end{proof}

Theorem \ref{th:main_general} and Proposition \ref{prp:adm_cover} then yield the following.
\begin{theo}
    Let $\Gamma$ be a metric graph whose canonical model contains at least four vertices and no multiple edges or separating vertices. Then the following are equivalent.
    \begin{itemize}
        \item[A.] $\Gamma$ is divisorially trigonal;
        \item[B.] $\Gamma$ is trigonal;
        \item[C.] $\Gamma$ admits a tropical admissible cover of degree $3$ to a metric tree.
    \end{itemize}
\end{theo}

\section{$D$-hyperelliptic graphs}\label{sc:hyperelliptic}
We are considering the most general case, that is any graph, not necessarily $3$-edge connected.
From \cite[Lemma 5.3]{BN}, \cite[Lemma 3.1]{MZ25}, any $3$-edge connected metric graph cannot be hyperelliptic. Since we are now dropping the assumption on $3$-edge connectivity, we may now consider also hyperelliptic graphs, for which the equivalence between the divisor of degree $2$ and rank $1$ and the non-degenerate harmonic morphism of degree $2$ to a metric tree has been proved in \cite{MC}.

Clearly hyperelliptic graphs are divisorially trigonal and we will now prove that they are also trigonal.

\begin{prop}\label{prp:hyp_trig}
Let $\Gamma$ be a hyperelliptic metric graph with canonical loopless model $(G_{-},l_{-})$ such that $|V(G_{-})|\geq3.$ 
Then $\Gamma$ is divisorially trigonal and trigonal.
\end{prop}
\begin{proof}
Let $H\in W_2^1(\Gamma)$ and define $D:=H+p$ with $p$ any point in $\Gamma.$ Clearly $\operatorname{deg}D=3$ and for any $E\in \operatorname{Div}(\Gamma);$ $E\geq 0$ and $\operatorname{deg}E=1,$ we have $D-E=H+p-E\sim E'+p$ with $E'$ effective since $H$ has rank $1$,
which yields that $D\in W_3^1(\Gamma)$, 
hence $\Gamma$ is divisorially trigonal.

We now show that there is tropical modification of $\Gamma,$ which admits a non-degenerate degree $3$ harmonic morphism to a tree $T$.

By \cite[Theorem 3.12]{MC}, there esist a non-degenerate harmonic morphism $\psi$ of degree $2$ from the canonical loopless model $(G_{-},l_{-})$ of $\Gamma$ to $(T,l_T)$ with $T$ a tree.

Let $t$ be any leaf of the tree and take any vertex $x \in \psi^{-1}(t).$ Define $(G,l)$ such that $G$ is obtained by attaching to $x$ the tree $T$ with edges of the same length as in $l_T,$ as shown in Figure \ref{fg:hyp}. 

\begin{figure}[ht]
\begin{tikzcd}
\begin{tikzpicture}
\draw (-0.2,0.2) node[anchor=east] {$G_{-}$};
    \draw(0,0)[red] to [out=90, in=150] (1,0);
    \draw(0,0)[red] to [out=270, in=210] (1,0);
    \vertex{0,0}
    \draw (1.25,0) node[anchor=south] {$2$};
    \vertex{1,0}
    \draw[thick,blue](1,0)--(1.5,0);
    \draw[verde](1.5,0)--(2.5,0.5);
    \draw[verde](1.5,0)--(2.5,-0.5);
    \vertex{1.5,0}\vertex{2.5,0.5}\vertex{2.5,-0.5}
    \draw(2.5,0.5) to [out=300, in=60] (2.5,-0.5);
    \draw(2.5,0.5) to [out=250, in=110] (2.5,-0.5);
    \end{tikzpicture}\arrow[d,"\psi"] &\begin{tikzpicture}
\draw (-0.2,0.2) node[anchor=east] {$G$};
    \draw(0,0)[red] to [out=90, in=150] (1,0);
    \draw(0,0)[red] to [out=270, in=210] (1,0);
    \vertex{0,0}
    \draw (1.25,0) node[anchor=south] {$2$};
    \vertex{1,0}
    \draw[thick,blue](1,0)--(1.5,0);
    \draw[verde](1.5,0)--(2.5,0.5);
    \draw[verde](1.5,0)--(2.5,-0.5);
    \vertex{1.5,0}\vertex{2.5,0.5}\vertex{2.5,-0.5}
    \draw(2.5,0.5) to [out=300, in=60] (2.5,-0.5);
    \draw(2.5,0.5) to [out=250, in=110] (2.5,-0.5);
    \draw[red](-0.2,0.5) -- (0.8, 0.5);
    \draw[blue](0.8,0.5) -- (1.3, 0.5);
    \draw[verde](1.3,0.5) -- (2.5, 0.5);
    \foreach \i in {-0.2,0.8,1.3,2.5} {
    	    \vertex{\i, 0.5}
    	    }  
    \end{tikzpicture}\arrow[d,"\varphi"]\\
    \begin{tikzpicture}
    \draw (-0.2,0.2) node[anchor=east] {$T$};
    \draw[red](0,0) -- (1, 0);
    \draw[blue](1,0) -- (1.5, 0);
    \draw[verde](1.5,0) -- (2.7, 0);
    \foreach \i in {0,1,1.5,2.7} {
    	    \vertex{\i, 0}
    	    }  
    \end{tikzpicture}&\begin{tikzpicture}
    \draw (-0.2,0.2) node[anchor=east] {$T$};
    \draw[red](0,0) -- (1, 0);
    \draw[blue](1,0) -- (1.5, 0);
    \draw[verde](1.5,0) -- (2.7, 0);
    \foreach \i in {0,1,1.5,2.7} {
    	    \vertex{\i, 0}
    	    }  
    \end{tikzpicture}\end{tikzcd}
    \caption{Non-degenerate harmonic morphisms of degree $2$ and $3$ from tropically equivalent metric graphs to the same tree.}\label{fg:hyp}
\end{figure}
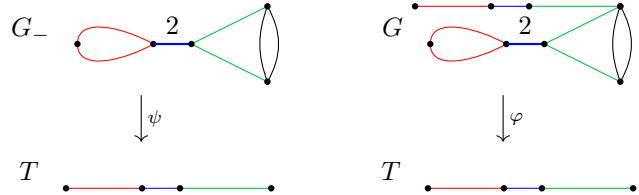

Clearly $(G,l)$ is a tropical modification of $(G_{-},l_{-}).$ Moreover, we extend $\psi$ to $\varphi$ such that $\varphi|_{(T,l_T)}$ is the identity.

No additional contraction has been defined, therefore $\varphi$ is non-degenerate. Moreover $\varphi$ is harmonic when restricted to both $(G_{-},l_{-})$ and $(T,l_T)$ and over their intersection, which by construction is a vertex that is sent to a leaf. 
Therefore the resulting morphism is again harmonic with degree increased by $1.$
\end{proof}

This agrees with the equivalence between the existence of a divisor of degree $3,$ rank $1$ on a metric graph $\Gamma$ and that of a non-degenerate harmonic morphism of degree $3$ from a a tropical modification of $\Gamma,$ to $(T,l_T)$ with $T$ a tree, which we have proved to hold in the $3$-edge connected case.

The difficult part, as in the previous cases, will be to define the morphism, starting from a divisor. We have already proved that this is true for hyperelliptic graphs, in Proposition \ref{prp:hyp_trig}. Therefore we will now consider non-hyperelliptic graphs.

Let us start by noticing that in this case, unlike the $3$-edge connected case, the morphism that we will construct, if it exists, does not have necessarily index $1$ over all non-contracted edges. An example is given in Figure \ref{fig:2_edge}.

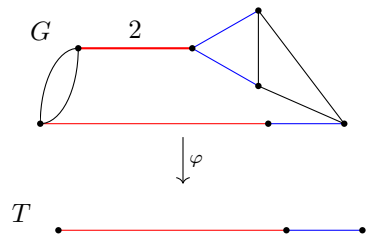
\begin{figure}[ht]
\begin{tikzcd}
\begin{tikzpicture}
        \draw (0.3,1.2) node[anchor=east] {$G$};
        \path[draw][thick,red] (0.5,1) -- (2, 1);
        \draw (1.25,1) node[anchor=south] {$2$};
        \path[draw][red] (0,0) -- (3,0);
    	\begin{scope}
		\clip (0,0) rectangle ++(0.5,1);
		\draw (0,1) ellipse (0.5 and 1);
        \draw (0.5,0) ellipse (0.5 and 1);
	  \end{scope}
   \vertex{0,0}
   \vertex{0.5,1}
        \path[draw][blue] (2,1) -- (2.87,1.5);
        \path[draw][blue] (2,1) -- (2.87,0.5);
        \path[draw][blue] (3,0) -- (4,0);
        \vertex{3,0}\vertex{4,0}
        \vertex{2,1}\vertex{2.87,1.5}\vertex{2.87,0.5}
        \draw (2.87,1.5) -- (2.87,0.5);
        \draw (2.87,0.5) -- (4,0);
        \draw (2.87,1.5) -- (4,0);
        \end{tikzpicture}\arrow[d,"\varphi"]\\
    \begin{tikzpicture}
    \draw (-0.2,0.2) node[anchor=east] {$T$};
    \draw[red](0,0) -- (3, 0);
    \draw[blue](3,0) -- (4, 0);
    \vertex{0,0}
    \vertex{3,0}
    \vertex{4,0}   
    \end{tikzpicture}\end{tikzcd}\caption{A non-degenerate harmonic morphism of degree $3$.}\label{fig:2_edge}
\end{figure}

Let us also consider a further example, depicted in Figure \ref{fig:hyp_block}.
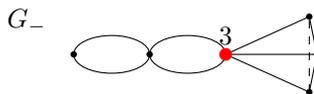
\begin{figure}[ht]
\begin{tikzcd}
\begin{tikzpicture}
        \draw (-0.7,0.45) node[anchor=east] {$G_{-}$};
        \draw (0,0) ellipse (0.5 and 0.25);
        \vertex{-0.5,0}\vertex{0.5,0}
    	\begin{scope}
		\clip (0.5,0) rectangle ++(1,0.5);
		\draw (1,0) ellipse (0.5 and 0.25);
	  \end{scope}
   \begin{scope}
		\clip (0.5,-0.5) rectangle ++(1,0.5);
		\draw (1,0) ellipse (0.5 and 0.25);
	  \end{scope}
   
        \path[draw][] (1.5,0) -- (2.7,0);
        \path[draw][] (1.5,0) -- (2.6,0.5);
        \path[draw][] (1.5,0) -- (2.6,-0.5);
        \vertex{2.7,0}\vertex{2.6,0.5}
        \vertex{2.6,-0.5}
        \draw[dashed] (2.6,-0.5) -- (2.6,0.5);
        \draw (2.6,-0.5) -- (2.7,0);
        \draw (2.6,0.5) -- (2.7,0);
        \divisor{1.5,0}
   \draw (1.5,0) node[anchor=south] {$3$};\end{tikzpicture}\end{tikzcd}\caption{A divisorially trigonal graph with no harmonic morphism of degree tree from its canonical loopless model $(G_{-},l_{-})$ to a tree.}\label{fig:hyp_block}
\end{figure}

Here, the metric graph is divisorially trigonal but there is no harmonic morphism of degree $3$ from its canonical loopless model to a tree. We observe however that its canonical loopless model contains a subgraph which is hyperelliptic (namely the two cycles on the left of the vertex supporting the divisor).
As in Proposition \ref{prp:hyp_trig}, one could define instead a harmonic morphism of degree $3$ from a tropical  modification of the canonical model, to a tree, as shown in Figure \ref{fig:hyp_block2}. 
\begin{figure}[ht]
\begin{tikzcd}
\begin{tikzpicture}
        \draw (-0.7,0.45) node[anchor=east] {$G$};
        \draw[red] (0,0) ellipse (0.5 and 0.25);
        \vertex{-0.5,0}
        \vertex{0.5,0}
    	\begin{scope}
		\clip (0.5,0) rectangle ++(1,0.5);
		\draw[blue] (1,0) ellipse (0.5 and 0.25);
	  \end{scope}
   \begin{scope}
		\clip (0.5,-0.5) rectangle ++(1,0.5);
		\draw[blue] (1,0) ellipse (0.5 and 0.25);
	  \end{scope}
   
        \path[draw][verde] (1.5,0) -- (2.7,0);
        \path[draw][verde] (1.5,0) -- (2.6,0.5);
        \path[draw][verde] (1.5,0) -- (2.6,-0.5);
        \path[draw][blue] (1.5,0) -- (0.6,0.5);
        \path[draw][red] (-0.3,1) -- (0.6,0.5);
        \vertex{0.6,0.5}\vertex{-0.3,1}
        \vertex{2.7,0}\vertex{2.6,0.5}
        \vertex{2.6,-0.5}
        \draw[dashed] (2.6,-0.5) -- (2.6,0.5);
        \draw (2.6,-0.5) -- (2.7,0);
        \draw (2.6,0.5) -- (2.7,0);
        \vertex{1.5,0}
\end{tikzpicture}\arrow[d]\\
    \begin{tikzpicture}
    \draw (-0.7,0.2) node[anchor=east] {$T$};
    \draw[red](-0.5,0) -- (0.5, 0);
    \draw[blue](0.5,0) -- (1.5, 0);
    \draw[verde](1.5,0) -- (2.7, 0);
    \vertex{-0.5,0}
    \vertex{0.5,0}
    \vertex{1.5,0}   
    \vertex{2.7,0}   
    \end{tikzpicture}\end{tikzcd}\caption{A non-degenerate harmonic morphism of degree $3$ from $(G,l)$ to $(T,l_T),$ with $T$ a tree and $G$ tropically equivalent to the graph in Figure \ref{fig:hyp_block}.}\label{fig:hyp_block2}
\end{figure}
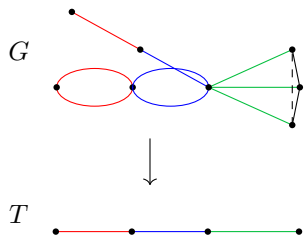
Such a morphism, restricted to the hyperelliptic part, is precisely the one determined by \cite[Theorem 3.12]{MC}.
Therefore we will need to take into account the presence of hyperelliptic subgraphs. 

We therefore now treat such hyperelliptic subgraphs, if there are any, and over such graphs a suitable morphism will be defined similarly as in Proposition \ref{prp:hyp_trig}. Then we will define the morphism over their complements in Subsection \ref{sc:non_hyp} and then finally prove that, under certain assumptions, the gluing of such morphisms still yields a non-degenerate harmonic morphism of degree $3,$ with a tree or a tree of triangles as target space.

Let us also recall that, from \cite[Corollaries 5.10,5.11]{BN}, the contraction of bridges does not changes the rank of a divisor. Therefore, given $D\in W_3^1(\Gamma)$, we will assume first $\Gamma$ to be $2$-edge connected and then extend later the construction on bridges.

From now on, we will assume $\Gamma$ non-hyperelliptic, $2$-edge connected, with canonical loopless model  $(G_{-},l_{-})$ such that $|V(G_{-})|>3.$

Given $D\in W_3^1(\Gamma),$ we want to explicitly describe a hyperelliptic subgraph $\Gamma_1\subset\Gamma$, if it exists, satisfying certain conditions.

\begin{defin}\label{def:D_hyp}
    Let $\Gamma$ be a metric graph with $D\in W_3^1(\Gamma).$ 
    A connected subcurve $\Gamma_1\subset\Gamma,$ with $|V(G_1)|\geq2,$ where $\Gamma_1=(G_1,l_1)$ with $G_1$ isomorphic to a subgraph of $G_{-},$ is said to be a $D$-\textbf{hyperelliptic half} of $\Gamma$ if there exists {$p\in\overline{\Gamma\setminus\Gamma_1}$} such that $D-p\sim H\in W_2^1(\Gamma_1)$ where
    \begin{enumerate}
        \item $H\sim_{\Gamma} x+y$ with $x\in\mathring{e_1},y\in\mathring {e_2}$ and $e_1,e_2\in E(G_1);$ $e_1\neq e_2$, where $\sim_{\Gamma}$ denotes linear equivalence in $\Gamma,$ and
        \item $\Gamma_1=\{\operatorname{Supp}(H')|H'\geq 0 \text{ and }H'\sim_{\Gamma} H\}$.
    \end{enumerate}
\end{defin} 

The following examples in Figures \ref{fg:D_half}, \ref{fg:D_half2} motivate the above definition and in particular the conditions (1) and (2). We provide examples of subcurves $\Gamma_1\subset\Gamma$ which do not satisfy both conditions in Definition \ref{def:D_hyp}. In fact, the construction that will be made in this section will not yield a harmonic morphism of degree $3$ for such graphs.

\begin{figure}[ht]
\centering
\begin{tikzcd}
\begin{tikzpicture}
    \draw (-0.2,1.2) node[anchor=east] {$\Gamma$};
    \draw[](0,0) --  (2,0);
    \draw[](0,1) --  (2,1);
    \draw[](0,-1) --  (2,-1);
    \draw[blue](0,0)--(0,1);\draw(0,0)--(0,-1);
    \draw(2,0)--(2,1);\draw(2,0)--(2,-1);
    \draw[] (0,-1) to [out=120, in=240] (0,0);
    \draw[blue] (0,0) to [out=120, in=240] (0,1);
    \draw[] (2,0) to [out=60, in=300]  (2,1);
    \draw[] (2,-1) to [out=60, in=300] (2,0);
    \foreach \i in {0,2} {
        \foreach \j in {0,1,-1}{
            \vertex{\i,\j}}
    }  
    \divisorBig{0,0}\divisorBig{0,1}
    \subgr{0,0}\subgr{0,1}\divisor{0,-1}
    \draw (0,-1) node[anchor=north] {$p$};
    \end{tikzpicture}
    \end{tikzcd}\caption{A metric subgraph $\Gamma_1$ (in blue) of $\Gamma,$ with $D\in W_3^1(\Gamma)$. Here $D-p=H\in W_2^1(\Gamma_1)$ and $\Gamma_1=\{\operatorname{Supp}(H')|H'\geq 0 \text{ and }H'\sim_{\Gamma} H\}$ but condition (1) in Definition \ref{def:D_hyp} is not satisfied.}\label{fg:D_half}
\end{figure}
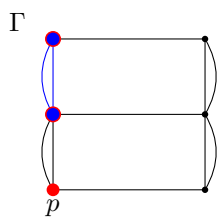

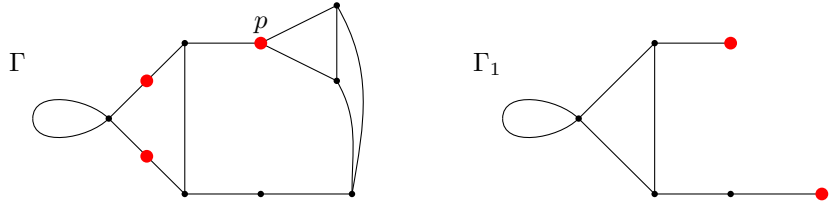
\begin{figure}[ht]
\centering
\begin{tikzcd}
    \begin{tikzpicture}
    \draw (-1.2,1) node[anchor=north] {$\Gamma$};
    \draw[](0,0) --  (1,1);
    \draw[](0,0) --  (1,-1);
    \draw[] (-1,0) to [out=90, in=135] (0,0);
    \draw[] (-1,0) to [out=270, in=225] (0,0);
    \draw(1,1) -- (1,-1);
    \vertex{0,0}
    \divisor{0.5,0.5}
    \divisor{0.5,-0.5}
    \draw (1,1)--(2,1);
    \draw (1,-1)--(2,-1);
    \foreach \i in {1,2} {
        \foreach \j in {1,-1}{
            \vertex{\i,\j}}
    }  
\draw (2,1)--(3,1.5);
\draw (2,1)--(3,0.5);
\draw (2,-1)--(3.2,-1);
\vertex{3,1.5}
\vertex{3,0.5}
\vertex{3.2,-1}
\divisor{2,1}
\draw (2,1) node[anchor=south] {$p$};
\draw(3,1.5)--(3,0.5);
\draw(3,0.5)to [out=300, in=90](3.2,-1);
\draw(3,1.5) to [out=300,in=80] (3.2,-1);
    \end{tikzpicture}&
        \begin{tikzpicture}
        \draw (-1.2,1) node[anchor=north] {$\Gamma_1$};
    \draw[](0,0) --  (1,1);
    \draw[](0,0) --  (1,-1);
    \draw[] (-1,0) to [out=90, in=135] (0,0);
    \draw[] (-1,0) to [out=270, in=225] (0,0);
    \draw(1,1) -- (1,-1);
    \vertex{0,0}
    
    \draw (1,1)--(2,1);
    \draw (1,-1)--(2,-1);
    \foreach \i in {1,2} {
        \foreach \j in {1,-1}{
            \vertex{\i,\j}}
    }  
\draw (2,-1)--(3.2,-1);
\vertex{3.2,-1}\divisor{3.2,-1}
    \divisor{2,1}
    \end{tikzpicture}
\end{tikzcd} \caption{A connected metric graph $\Gamma_1\subset \Gamma$ with $D\in W_3^1(\Gamma)$ and $D-p=H\in W_2^1(\Gamma_1)$. Here $\Gamma_1=\{\operatorname{Supp}(H')|H'\geq 0 \text{ and }H'\sim_{\Gamma_1} H\}$ but $\Gamma_1\supsetneq\{ \operatorname{Supp}(H)|H'\geq 0 \text{ and }H'\sim_{\Gamma} H\},$ hence condition (2) in Definition \ref{def:D_hyp} is not satisfied.}\label{fg:D_half2}
\end{figure}

\begin{rk}\label{rk:hyp_loops}
 With the above definition, we will also consider, as $D$-hyperelliptic halves, portion of loops.
 For instance, consider the example represented in Figure \ref{fg:loop_hyp}. The subgraph $\Gamma_1$ does satisfy the conditions in Definition \ref{def:D_hyp}, hence it is a $D$-hyperelliptic half. Moreover, notice that, because on the condition on the vertices, any $\Gamma_1$ of this type consists of an edge whose length is bigger or equal than half of that of the entire loop.  
     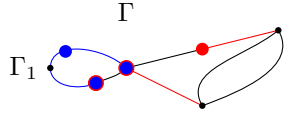
\begin{figure}[ht]
\centering
\begin{tikzcd}
    \begin{tikzpicture}
    \draw(0,0) --  (1,0.25);
    \draw[red](2,0.5) --  (1,0.25);
    \draw[red](0,0) --  (1,-0.5);
    \draw[blue] (-1,0) to [out=90, in=135] (0,0);
    \draw[blue] (-1,0) to [out=270, in=205] (-0.4,-0.2);
    \subgr{-0.8,0.22}
    \draw[] (-0.4,-0.2) to [out=25, in=225] (0,0);
    \draw (0,1) node[anchor=north] {$\Gamma$};
    \draw (-1,0) node[anchor=east] {$\Gamma_1$};
    \divisorBig{0,0}
    \divisor{1,0.25}
    \divisorBig{-0.4,-0.2}
    \draw(2,0.5)to [out=300, in=20](1,-0.5);
    \draw(2,0.5) to [out=220,in=120] (1,-0.5);
    \subgr{0,0}\subgr{-0.4,-0.2}
    \vertex{-1,0}
    \vertex{2,0.5}\vertex{1,-0.5}
    \end{tikzpicture}
\end{tikzcd} \caption{The metric graph $\Gamma$ is such that the red edges have same length, and the red points define a divisor $D$ of degree $3$ and rank $1$.}\label{fg:loop_hyp}
\end{figure}
\end{rk}

Notice also that we are assuming $\Gamma$ non-hyperelliptic, therefore the complement of any $D$-hyperellptic half cannot be empty.

Let us also add few results on the properties of $D$-hyperelliptic halves.

\begin{rk}
    \label{rk:2cut}
    Let $\Gamma_1=(G_1,l_1)$ be a $D$-hyperellptic half, which does not consists only of a portion of loop. Then condition (1) in Definiton \ref{def:D_hyp} can be replaced with the following:
    \begin{itemize}
    \item[(1a)] $H\sim_{\Gamma} x+y$ with $x\in\mathring{e_1},y\in\mathring {e_2}$ and $\{e_1,e_2\}$ is a $2$-edge cut of $\Gamma.$
    \end{itemize}
    
    Indeed, by condition (1) in Definiton \ref{def:D_hyp} we have $D\sim p+x+y,$ with $x,y$ contained in the interior of two distinct edges $e_1,e_2$ in $\Gamma_1.$ Using Dhar's burning algorithm it is easy to see that such two edges form a $2$-edge cut of $G_1.$

    If $\Gamma_1$ doesn't consist only of a portion of a loop, then $G_1$ contains more edges than the set $\{e_1,e_2\}.$ Let $e\in E(G_1)$ with $e\neq e_1,e_2$ and pick $y'\in e.$ Then $x+y\sim_{\Gamma}x'+y'$ for some $x'\in \Gamma_1,$ by condition (2) in Definition \ref{def:D_hyp}. Then, by \cite[Lemma 3.1. (1)]{MC} the edges containing $x,y$ must form a $2$-edge cut of $\Gamma.$
\end{rk}

Any $D$-hyperelliptic half $\Gamma_1$ is clearly hyperelliptic, and thus admits a non-degenerate degree $2$ harmonic morphism $\varphi$ to a tree $T_1.$ 
If $\Gamma_1$ is strictly contained in a loop as in Remark \ref{rk:hyp_loops} the morphism can be constructed by adding a vertex at the midpoint of $\Gamma_1$ (and eventually removing that at the midpoint of the entire loop) and identifying the two halves obtained by this refinement.

We attach a copy of $T_1$ at $p,$ where $p\in\overline{\Gamma\setminus\Gamma_1}$ such that $D\sim H+p$ with $H\in W_2^1(\Gamma_1)$. The resulting graph $\Gamma'$ is clearly tropically equivalent to $\Gamma$ and we define the morphism over $\Gamma_1':=\Gamma_1\cup T_1$ by sending each edge of the tree to its copy in the target space, with index 1. 
Unlike the construction in Proposition \ref{prp:hyp_trig}, the added tree cannot be attached to any preimage of a leaf, but its position will be determined by $p.$

Repeating such a construction for any $D$-hyperelliptic half $\Gamma_i\subset \Gamma$ yields a morphism from the union of $D$-hyperelliptic halves with the corresponding trees $\bigcup_i \Gamma_i'$ with $\Gamma_i':=\Gamma_i\cup T_i,$ that we will denote by $\varphi_D^\text{hyp}$.
Such a morphism, however might not even be well defined. We will now prove that if all $\Gamma_i$ satisfy the following properties, then $\varphi_D^\text{hyp}$, is a non-degenerate degree $3$ harmonic morphism to (a disjoint union) of trees.

\begin{itemize}
    \item[(H1)] $\Gamma_i\cap\overline{\Gamma\setminus\Gamma_i}=\{x_0,y_0\}$ where $D-p\sim H$ as in Definition \ref{def:D_hyp}, with $H\sim x_0+y_0,$ (where $x_0$ and $y_0$ might be equal).
    \item[(H2)] If $\Gamma_i\neq\Gamma_j$, then $|\Gamma_i\cap\Gamma_j|\leq 1.$ 
\end{itemize}

We will later prove in subsection \ref{ssc:hyp_necklace} that if $\Gamma$ is not a necklace, then (H1) and (H2) hold for any $D$-hyperelliptic half $\Gamma_i$.

\begin{prop}\label{prp:morph_hyp}
    Let $\Gamma$ be a $2$-edge connected non-hyperelliptic metric graph with $D\in W_3^1(\Gamma)$ and $D$-hyperelliptic halves $\Gamma_1,\Gamma_2$ satisfying \emph{(H1)}, \emph{(H2)}. Let $\varphi_D^\text{\emph{hyp}}$ be the morphism constructed above, then $\varphi_D^\text{hyp}|_{\Gamma_1'\cup\Gamma_2'}$ is non-degenerate harmonic morphism of degree $3$ to a tree (or a disjoint union of trees).
\end{prop}

\begin{proof}
First of all, we observe that the morphism is well defined: by (H2), the subgraphs $\Gamma_1$ and $\Gamma_2$ intersect at one point at most, therefore the morphisms over each $D$-hyperelliptic half does not affect the edges contained in the other.

Notice also that the gluing of two trees over a point is still a tree and that, when gluing the two morphisms, no contraction is added, therefore the morphism is still non-degenerate and we only need to verify the harmonicity.
By construction, $\varphi_D^\text{hyp}|_{\Gamma_i'}$ is harmonic for $i=1,2.$ Therefore we only have to prove that the morphism is harmonic over the intersection, which consists of (at most) a unique point $\{x\},$ again by property (H2).

Let $G$ be the graph isomorphic to a subgraph of $G_{-}$ such that $\Gamma_1'\cup\Gamma_2'=(G,l_G).$ If $x\in \mathring e;$ $e\in E(G),$ then there is nothing to prove: we have $\mu_{\varphi_D^\text{hyp}}(e)=1$ (recall that $\Gamma_i$ are $2$-edge connected) for any $e$ such that $\varphi_D^\text{hyp}(e)\in E_{{\varphi_D^\text{hyp}}(x)}(T)$. Therefore let us assume $x\in V(G).$ Let $t:=\varphi_D^\text{hyp}(x),$ we need to show that the quantity $$\sum_{\substack{e\in E_x(G)\\ \varphi_D^\text{hyp}(e)=e'}} \mu_{\varphi_D^\text{hyp}}(e)$$ is constant for any $e'\in E_t(T).$

Since $\mu_{\varphi_D^\text{hyp}}(e)=1$ for any $e$ such that $\varphi_D^\text{hyp}(e)\in E_t(T),$ the above quantity only depends on $D'(x)$ for some $D'\sim D.$ 
By (H1), we have that $D\sim D_1\sim D_2$ with $D_i=x+y_i+p_i$ and $x+y_i\sim H_i\in W_2^1(\Gamma_i);$ $i=1,2.$ 
If $y_i,p_i\neq x,$ then harmonicity follows.
If instead $D_1(x)\geq2,$ then we would have, for instance, $2x+p_1\sim x+y_2+p_2.$ This implies in particular that the valence in $\Gamma_1$ over $x$ is $2$ and $D_2(x)\geq 2,$ otherwise, starting Dhar's burning algorithm from any point in $\Gamma\setminus (\Gamma_1\cup\Gamma_2)$ or in $\Gamma_1$ (depending on where $p_2$ does not lie), burns the whole support of $D_2,$ hence the whole graph. The fact that $D_1(x)=D_2(x)$ follows by $2$-edge connectivity, hence $\varphi_D^{\text{hyp}}$ is harmonic.
\end{proof}

\begin{rk}
    Using the arguments in \cite[Corollary 3.11, Theorem 3.12]{MC} the construction can be easily generalised to $\Gamma$ non necessarily $2$-edge connected. More precisely, for any connected graph, we define a $D$-hyperelliptic graph if the graph obtained by contracting all bridges is $D$-half hyperelliptic as in Definition \ref{def:D_hyp} and then generalize the construction of the morphism as done in \cite[Theorem 3.12]{MC}. 

    Over the intersection points of distinct hyperelliptic halves, we can also extend the morphism, as follows.
    First of all, let us denote by $\Gamma'$ the graph obtained by the contraction of a bridge $b$ and denote by $v$ the image of $b$ via such a contraction, which we will assume to be the intersection point of two distinct $D$-hyperelliptic halves.

    Then apply the constuction over the two $D$-hyperelliptic halves, as in Figure \ref{fg:bridge}.

      \begin{figure}[ht]
\centering
\begin{tikzcd}
    \begin{tikzpicture}
    \draw(0,0) --  (1,0.25);
    \draw[red](2,0.5) --  (1,0.25);
    
    \draw[red](0,0) --  (1,-0.5);
    \draw[blue] (-1,0) to [out=90, in=135] (0,0);
    \draw[blue] (-1,0) to [out=270, in=205] (-0.4,-0.2);
    \draw[red](-0.4,-0.2)--(0.6,-0.7);
    \vertex{0.6,-0.7}
    \draw[blue](1,0.25)--(0.2,0.5);
    \vertex{0.2,0.5}
    \draw[] (-0.4,-0.2) to [out=25, in=225] (0,0);
    \draw (0,1) node[anchor=north] {$\Gamma'$};
    \divisor{0,0}
    \divisor{1,0.25}
    \divisor{-0.4,-0.2}
    \draw(2,0.5)to [out=300, in=20](1,-0.5);
    \draw(2,0.5) to [out=220,in=120] (1,-0.5);
    \vertex{2,0.5}\vertex{1,-0.5}
    \vertex{-0.8,0.225}
    \draw[->] (0.2,-1) to (0.2,-1.5);
    \draw[blue](-0.7,-2)--(0.2,-2);
    \draw[red](1.3,-2)--(0.2,-2);
    \vertex{-0.7,-2}\vertex{0.2,-2}\vertex{1.3,-2}
    \end{tikzpicture}
\end{tikzcd} \caption{A non-degenerate harmonic morphism of degree $3$ from a tropical modification of $\Gamma'$ to a tree.}\label{fg:bridge}
\end{figure}
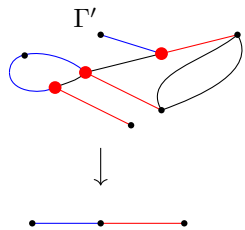

Then, in order to extend the morphism also over the bridge, it is sufficient to insert an exact copy of the bridge on the added trees via the previous construction, at the vertices in the same pre-image as $v.$ An example is provided in Figure \ref{fg:bridge1} and clearly the resulting morphism has the same properties.
     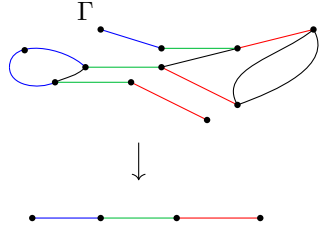
\begin{figure}[ht]
\centering
\begin{tikzcd}
    \begin{tikzpicture}
    \draw(1,0) --  (2,0.25);
    \draw[red](3,0.5) --  (2,0.25);
    \draw[verde](0,0)--(1,0);
    \draw[red](1,0) --  (2,-0.5);
    \draw[blue] (-1,0) to [out=90, in=135] (0,0);
    \draw[blue] (-1,0) to [out=270, in=205] (-0.4,-0.2);
    \draw[verde](-0.4,-0.2)--(0.6,-0.2);
    \draw[red](0.6,-0.2)--(1.6,-0.7);
    \vertex{1.6,-0.7}\vertex{0.6,-0.2}
    \vertex{1,0}
    \draw[blue](1,0.25)--(0.2,0.5);
    \vertex{0.2,0.5}
    \draw[] (-0.4,-0.2) to [out=25, in=225] (0,0);
    \draw (0,1) node[anchor=north] {$\Gamma$};
    \draw[verde](1,0.25)--(2,0.25);
    \vertex{0,0}
    \vertex{1,0.25}\vertex{2,0.25}
    \vertex{-0.4,-0.2}
    \draw(3,0.5)to [out=300, in=20](2,-0.5);
    \draw(3,0.5) to [out=220,in=120] (2,-0.5);
    \vertex{3,0.5}\vertex{2,-0.5}
    \vertex{-0.8,0.225}
    \draw[->] (0.7,-1) to (0.7,-1.5);
    \draw[blue](-0.7,-2)--(0.2,-2);
    \draw[verde](1.2,-2)--(0.2,-2);
    \draw[red](1.2,-2)--(2.3,-2);
    \vertex{-0.7,-2}\vertex{0.2,-2}\vertex{1.2,-2}\vertex{2.3,-2}
    \end{tikzpicture}
\end{tikzcd} \caption{A non-degenerate harmonic morphism of degree $3$ from a tropical modification of $\Gamma$ to a tree.}\label{fg:bridge1}
\end{figure}
\end{rk}

In the following subsection we will characterize metric graphs for which properties (H1), (H2) are satisfied, while the construction of the morphism over the complement of the union of $D$-hyperelliptic halves will be carried out in Section \ref{sc:non_hyp}.

\subsection{Necklaces of hyperelliptic graphs}\label{ssc:hyp_necklace}

As discussed in Example \ref{ex:neck}, the morphism $\varphi_D^{\text{hyp}}$ might not extend to a well-defined degree $3$ non-degenerate harmonic morphism over the union of $D$-hyperelliptic halves.

As anticipated, this might happen when properties (H1), (H2) are not satisfied. We will here show that if (H1), (H2) are not satisfied then $\Gamma$ must be a necklace.

\begin{ex}\label{ex:D_hyp_intersection}
Let us consider first a necklace of loops, i.e. a necklace whose connected components obtained by removing the cycle are all loops, as represented in Figure \ref{fg:necklace_loops}.

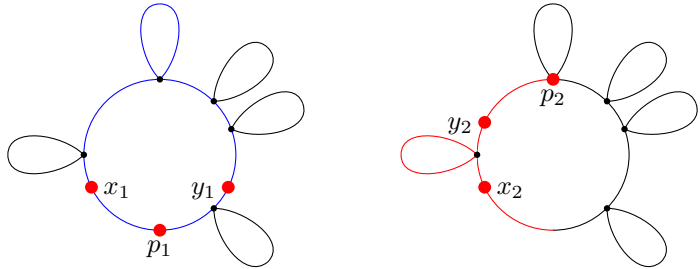
\begin{figure}[ht]\begin{tikzcd}
\begin{tikzpicture}
    \draw[blue](0,0) circle (1);
    \draw (-0.90,-0.43) node[anchor= west] {$x_1$};
    \draw (0.90,-0.43) node[anchor= east] {$y_1$};
    \draw (0,-1) node[anchor=north] {$p_1$};
    \draw[] (-2,0) to [out=90, in=135] (-1,0);
    \draw[] (-2,0) to [out=270, in=225] (-1,0);
    \draw[blue] (0,2) to [out=0, in=45] (0,1);
    \draw[blue] (0,2) to [out=180, in=135] (0,1);
    \draw[] (1.42,1.42) to [out=315, in=0] (0.71,0.71);
    \draw[] (1.42,1.42) to [out=135, in=90] (0.71,0.71);
    \draw[] (1.88,0.68) to [out=290, in=335] (0.94,0.34);
    \draw[] (1.88,0.68) to [out=110, in=65] (0.94,0.34);
    \draw[] (1.42,-1.42) to [out=220, in=270] (0.71,-0.71);
    \draw[] (1.42,-1.42) to [out=45, in=0] (0.71,-0.71);
    \vertex{-1,0}\vertex{0,1}\vertex{0.71,0.71}\vertex{0.94,0.34}\vertex{0.71,-0.71}
    \divisor{-0.90,-0.43}\divisor{0.90,-0.43}\divisor{0,-1}
    \end{tikzpicture}  &
    \begin{tikzpicture}
    \begin{scope}
    \clip (0,-1) rectangle (1,1);
    \draw[](0,0) circle (1);
    \end{scope}
    \begin{scope}
    \clip (-1,-1) rectangle (0,1);
    \draw[red](0,0) circle (1);
    \end{scope}
    \draw (-0.90,-0.43) node[anchor= west] {$x_2$};
    \draw (-0.90,0.43) node[anchor= east] {$y_2$};
    \draw (0,1) node[anchor=north] {$p_2$};
    \draw[red] (-2,0) to [out=90, in=135] (-1,0);
    \draw[red] (-2,0) to [out=270, in=225] (-1,0);
    \draw[] (0,2) to [out=0, in=45] (0,1);
    \draw[] (0,2) to [out=180, in=135] (0,1);
    \draw[] (1.42,1.42) to [out=315, in=0] (0.71,0.71);
    \draw[] (1.42,1.42) to [out=135, in=90] (0.71,0.71);
    \draw[] (1.88,0.68) to [out=290, in=335] (0.94,0.34);
    \draw[] (1.88,0.68) to [out=110, in=65] (0.94,0.34);
    \draw[] (1.42,-1.42) to [out=220, in=270] (0.71,-0.71);
    \draw[] (1.42,-1.42) to [out=45, in=0] (0.71,-0.71);
    \vertex{-1,0}\vertex{0,1}\vertex{0.71,0.71}\vertex{0.94,0.34}\vertex{0.71,-0.71}
    \divisor{-0.90,-0.43}\divisor{-0.90,0.43}\divisor{0,1}
    \end{tikzpicture}    
\end{tikzcd}\caption{A necklace of loops $\Gamma$ with a divisor $D\in W_3^1(\Gamma)$ and two $D$-hyperelliptic halves.}\label{fg:necklace_loops}
\end{figure}

Let $D=x_1+y_1+p_1\sim x_2+y_2+p_2$ defined as in the figure. Then one can check that $D\in W_3^1(\Gamma)$ and that the subgraphs $\Gamma_1,\Gamma_2$ colored in blue and red are $D$-hyperelliptic half, whose intersection does not consists of a single point.
\end{ex}

\begin{prop}
    Let $\Gamma$ be a $2$-edge connected, non-hyperelliptic metric graph with $D\in W_3^1(\Gamma).$ If $\Gamma$ is not a necklace, then any $D$-hyperelliptic half in $\Gamma$ satisfies (H1), (H2).
\end{prop}

The proof of this proposition is a direct consequence of the following lemmas.

\begin{lem}\label{lm:necklace1}
    Let $\Gamma_1\subset\Gamma$ be a $D$-hyperelliptic half, with $D\sim H+p$ and $H=x+y$ supported in the interior of a $2$-edge cut $\{e_1,e_2\}$. Consider the two connected components of $\Gamma\setminus\{e_1,e_2\}.$ If none of them is entirely contained in $\Gamma_1$ then $\Gamma$ is a necklace.
\end{lem}

\begin{proof}
    By definition of $D$-hyperelliptic half, we have $D\sim H+p$ with $p\in\overline{\Gamma\setminus \Gamma_1}.$ Then clearly the connected component of $\Gamma\setminus\{e_1,e_2\}$ containing $p,$ will never be entirely contained in $\Gamma_1$ unless $\Gamma=\Gamma_1,$ which is impossible since $\Gamma$ is assumed to be non-hyperelliptic.
    
    Assume then that also the other component $\Gamma'$ is not entirely contained in $\Gamma_1$: there is a point $z\in\Gamma'$ such that $z\notin\Gamma_1.$ This means that $z\not\in\operatorname{Supp}(H'),$ where $H'\sim H,$ i.e. $x+y-z$ cannot be linearly equivalent to an effective divisor, hence there exist $x',y'$ such that $x+y\sim x'+y',$ and Dhar's burning algorithm burns the whole graph if we consider the divisor $x'+y'$ and start a fire at $z.$

    Since $D$ has rank $1,$ the divisor $x'+y'+p-z$ must be linearly equivalent to an effective divisor. In particular, if we consider $D\sim x'+y'+p$ and start a fire at $z,$ then the graph cannot burn all. 

    This means that the fire has to burn precisely one between $x',y'$ and then stop at $p.$ Say the fire burns $x'$ and stops at $y',p.$ 
    Then in particular there is a unique path in $\Gamma'$ from $x'$ to $y'.$ 
    
    Moreover, the last two edges of the burnt paths stopping at $y',p$ respectively, must form a $2$-edge cut. We can then consider the rational function with slope $-1$ along such an edge cut, from $y',p$ towards the burnt edges, until the first path reaches a vertex of the canonical model. 
    Any such vertex must be a separating vertex otherwise we would have two incoming burning edges which would burn the vertex. Notice that this is true also in the case the two paths reach the vertex at the same time. The two vertices cannot be connected via an edge since otherwise $x',y'$ would not determine a $2$-edge cut: this second edge would determine a second path in $\Gamma'$ connecting $x'$ to $y'.$
    
    The same argument can now be applied to the other component of $\Gamma\setminus\{e_1,e_2\},$ by moving $y',p'$ towards the unburnt part of the graph, which proves that $\Gamma$ is a necklace.
\end{proof}

\begin{lem}\label{lm:necklace2}
    Let $\Gamma_1,\Gamma_2\subset\Gamma$ be two distinct $D$-hyperelliptic halves.
    If they meet in their interior, then $\Gamma$ is a necklace.
\end{lem}

\begin{proof}
    Let $e\subset\Gamma_1\cap\Gamma_2$ be an edge (of a refinement of $\Gamma_1$ and $\Gamma_2$). By definition, there is a point $x\in e$ such that $D\sim x+y_1+p_1\sim x+y_2+p_2$ with $y_i\in \Gamma_i$ and $p_i\in\overline{\Gamma\setminus \Gamma_i}$ for $i=1,2.$ If $y_1,y_2\in\Gamma_1\cap\Gamma_2$ then $\Gamma_1=\Gamma_2$ from \cite[Theorem A.1]{L}.
    By the same argument we may further assume that $\Gamma_1,\Gamma_2$ are not portions of the same loop. By Remark \ref{rk:hyp_loops} they both contain the same half of the length of the loop, hence we can always find $y_1',y_2'\in\Gamma_1\cap\Gamma_2$ such that $x+y_1'\sim x+y_2'$. 
    
    Then recall that from Remark \ref{rk:2cut} we may assume $y_i$ in the interior of an edge $e_i$ forming a $2$-edge cut with $e,$ and from \cite[Theorem A.1]{L}, we may assume $e_i\not\subset\Gamma_j$, with $i\neq j$.

    Then we can write
    $$\Gamma\setminus\{e,e_1\}= \Gamma_1^{(1)}\sqcup \Gamma_1^{(2)},$$
    $$\Gamma\setminus\{e,e_2\}= \Gamma_2^{(1)}\sqcup \Gamma_2^{(2)},$$ with $p_i\in\Gamma_i^{(2)}.$

    Let $\Gamma_0:= \Gamma_1^{(1)}\setminus {e_2}=\Gamma_2^{(1)}\setminus {e_1}.$    
    If by contradiction $\Gamma$ is not a necklace, then by Lemma \ref{lm:necklace1} we have that $\Gamma_0$ is entirely contained in both $\Gamma_1$ and $\Gamma_2.$ Moreover, by construction we may assume that $\Gamma_0$ contains pairs of points both in the support of $H_1,H_2$ where $x+y_i\sim H_i.$ Then by \cite[Theorem A.1]{L} $H_1\sim H_2$ and by definition of $D$-hyperelliptic half then $\Gamma_1=\Gamma_2,$ giving a contradiction.
\end{proof}

From now on, we will consider graphs which are not necklaces, therefore, as a consequence of Lemmas \ref{lm:necklace1} and \ref{lm:necklace2} we have that for any $D$-hyperelliptic half $\Gamma_1$, removing a $2$-edge cut supporting the divisor $H\in W_2^1(\Gamma_1)$ gives two components, one entirely contained in $\Gamma_1.$ Moreover, two distinct $D$-hyperelliptic halves cannot meet in their interior. Such properties will allow us to prove that in particular two distinct $D$-hyperelliptic halves meet at at most one point. This will be fundamental to prove that the gluing of the harmonic morphisms of degree $2$ constructed over two distinct hyperelliptic halves, as in Proposition \ref{prp:hyp_trig}, is still harmonic.

\begin{lem}\label{lm:gluing}
    Let $\Gamma_1$ be a $D$-hyperelliptic half in $\Gamma$, which is not a necklace. Then $\Gamma_1\cap\overline{\Gamma\setminus\Gamma_1}=\{x_0,y_0\}$ and $D-p\sim H$ as in Definition \ref{def:D_hyp}, with $H\sim x_0+y_0,$ (where $x_0$ and $y_0$ might be equal).

\end{lem}
\begin{proof}
    Let $x_0\in \Gamma_1\cap\overline{\Gamma\setminus\Gamma_1},$ then, in order to prove the claim, it suffices to show that $x_0+x\sim H,$ $\forall x\in (\Gamma_1\cap\overline{\Gamma\setminus\Gamma_1})\setminus\{x_0\},$ or that $2x_0\sim H$ if $\{x_0\}= \Gamma_1\cap\overline{\Gamma\setminus\Gamma_1}.$
    Indeed, this implies that $\Gamma_1\cap\overline{\Gamma\setminus\Gamma_1}$ consists of $x_0$ and the unique point $y_0$ such that $x_0+y_0\sim H$, determined by the hyperelliptic involution.

    The statement is clearly true when $\Gamma_1$ consists of a portion of a loop. Therefore let us assume that it is not, and 
    assume by contradiction that $H\sim x_0+x,$ with $x\not\in\Gamma_1\cap\overline{\Gamma\setminus\Gamma_1}$. 
    By Remark \ref{rk:2cut}, $x_0,x$ define a $2$-edge cut $\mathcal E\subset \Gamma,$ of edges in $\Gamma_1.$ 
    By $2$-edge connectivity, $\Gamma$ with the cut removed consists of two components $C_0$, $C_1$, with $C_1\subset\Gamma_1$, as a consequence of Lemma \ref{lm:necklace1}. Moreover, since $x_0\in\Gamma_1\cap\overline{\Gamma\setminus\Gamma_1}$ and $C_1\subset\Gamma_1,$ then $x_0,x\in C_0.$ Take $y\in\Gamma_1\cap C_0,$ along a path starting from $x,$ entirely contained in $\Gamma_1.$ 
    Such a point exists since we are under the assumption that $x\not\in\Gamma_1\cap\overline{\Gamma\setminus\Gamma_1}.$

    Then consider the linear equivalence $x_0+x\sim_{\Gamma} y+y'$ for some $y'\in \Gamma_1.$ 
    If $y'\in C_1\cup \mathcal E$ then, from \cite[Lemma 3.1 (i)]{MC}, there is a path $yy'$ over which $f$ is constant. However such a path has to cross either $x_0$ or $x_1$, over which the total slope is non-trivial. Therefore $y'\in C_0,$ and the path $x_0y$ does not pass through $x,$ otherwise we would have again that the path $yy'$ crosses either $x_0$ or $x.$
    This contradicts $x_0\in\Gamma_1\cap\overline{\Gamma\setminus\Gamma_1}.$
\end{proof}

\begin{lem}\label{lm:inters}
    Let $\Gamma_1,$ $\Gamma_2$ be two distinct $D$-hyperelliptic halves in $\Gamma$, which is not a necklace. Then $|\Gamma_1\cap\Gamma_2|\leq1.$
\end{lem}

\begin{proof}
    Let us first notice that if $\Gamma_1\cap\Gamma_2\neq\emptyset,$ $\Gamma_1\not\subseteq\Gamma_2$ and $\Gamma_2\not\subseteq\Gamma_1,$ then $(\Gamma_1\cap \overline{\Gamma\setminus\Gamma_1})\cap(\Gamma_2\cap \overline{\Gamma\setminus\Gamma_2})\neq\emptyset$. Indeed, a consequence of Lemma \ref{lm:necklace2} is that $\Gamma_1,$ $\Gamma_2$ cannot meet in their interiors.

    Denote by $i_j$ the hyperelliptic involution on $\Gamma_j,$ $j=1,2.$
    Take $x_0\in (\Gamma_1\cap \overline{\Gamma\setminus\Gamma_1})\cap (\Gamma_2\cap \overline{\Gamma\setminus\Gamma_2})$ and
    by Lemma \ref{lm:gluing}, we have $x_1=i_1(x_0)\in \Gamma_1\cap \overline{\Gamma\setminus\Gamma_1}$ and $x_2=i_2(x_0)\in \Gamma_2\cap \overline{\Gamma\setminus\Gamma_2}.$ If $x_1=x_2$ then we would have $x_0+x_1+p_1\sim x_0+x_2+p_2$ hence $p_1\sim p_2$ which contradicts $\Gamma_1\neq\Gamma_2.$
    Assume then $x_1\neq x_2.$
    If $x_2\in \Gamma_1,$ then there exists $i_1(x_2)\in\Gamma_1$ such that $x_0+x_1\sim x_2+i_1(x_2).$ This defines a rational function $f$ which is constant on any path on the complement of $\Gamma_1.$ Among such path there is one connecting $x_0$ and $x_2,$ otherwise we would have $\Gamma_2\subseteq\Gamma_1,$ which is not possible by the above remark. This gives a contradiction: $x_0$ lies in the set $\mathbf m$ over which $f$ is minimized $x_2\in\mathbf M,$ the set over which $f$ is maximized. 
    Then $x_2\notin \Gamma_1$ and, by symmetry, we also have $x_1\notin\Gamma_2.$

    Finally, assume by contradiction that there exists $y_0\in\Gamma_1\cap\Gamma_2\setminus\{x_0\}$ and consider the linear equivalence $x_0+x_1\sim_{\Gamma} y_0+y_1,$ for some $y_1\in\Gamma_1$. This now defines a rational function $g$ which achieves its minimum on the complement of $\Gamma_1,$ hence on $x_2.$ 
    
    By connectivity of $\Gamma_2$, there is a path $x_2y_0$ in $\Gamma_2,$ which is not entirely contained in $\Gamma_1$ since $x_2\notin\Gamma_1$. 
    Since $g$ is minimized on $x_2$ and maximized on $y_0,$ then the path $x_2y_0$ must pass through $x_0$ or $x_1.$ However such a path cannot pass through $x_0$ since $x_0\in\Gamma_2\cap \overline{\Gamma\setminus\Gamma_2}$ but also it cannot pass through $x_1$ since $x_1\not\in\Gamma_2,$ giving a contradiction.
\end{proof}

Let us now consider the case where $\Gamma$ is a necklace. We have claimed that in general in this case it is not possible to construct a non-degenerate harmonic morphism of degree $3$ to a tree or a triangle of trees, but there might be some exceptions.

\begin{ex}\label{ex:circle_loops}
We consider again a necklace of loops. As discussed in Example \ref{ex:neck}, the graph is divisorially trigonal and we show that it is also trigonal.

Each of the loops in the necklace is contained in a $D$-hyperelliptic half, namely the union of the loop and the two paths contained in the cycle with starting point $x$ of length $d(x,p),$ where $2x+p\sim D$, and $x$ is the vertex over which the loop is glued. However the above construction made for $D$-hyperelliptic halves for $\Gamma$ not a necklace will not yield a well defined harmonic and degree $3$ morphism to a tree.

Instead, we can consider the usual construction for the morphism over a single loop: add a vertex at the midpoint of the loop and add a leg of length half of that of the loop at $p$ and sent these three edges, of equal length, to an edge of same length. Then consider the two distinct paths from $x$ to $p$: divide each path in two edges, the closest to $x$ of length equal to a third of the whole path in which it is contained. Then send this edge with multiplicity $2$ and the other edge completing the path with multiplicity $1$ to the same edge of the three. Contract instead all the other loops. The resulting morphism, represented in Figure \ref{fg:circle_loops3}, is non-degenerate, harmonic and of degree $3.$

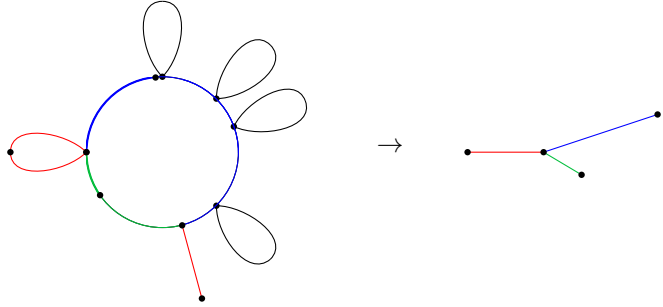
\begin{figure}[ht]\begin{tikzcd}
    \begin{tikzpicture}
    \draw(0,0) circle (1);
    \draw[red] (-2,0) to [out=90, in=135] (-1,0);
    \draw[red] (-2,0) to [out=270, in=225] (-1,0);
    
    \draw[] (0,2) to [out=0, in=45] (0,1);
    \draw[] (0,2) to [out=180, in=135] (0,1);
    \draw[] (1.42,1.42) to [out=315, in=0] (0.71,0.71);
    \draw[] (1.42,1.42) to [out=135, in=90] (0.71,0.71);
    \draw[] (1.88,0.68) to [out=290, in=335] (0.94,0.34);
    \draw[] (1.88,0.68) to [out=110, in=65] (0.94,0.34);
    \draw[] (1.42,-1.42) to [out=220, in=270] (0.71,-0.71);
    \draw[] (1.42,-1.42) to [out=45, in=0] (0.71,-0.71);
    \vertex{-1,0}\vertex{0,1}\vertex{0.71,0.71}\vertex{0.94,0.34}\vertex{0.71,-0.71}
    \vertex{-2,0}
    \draw[red] (0.26,-0.97)--(0.52,-1.94);
    \vertex{0.52,-1.94}
    \begin{scope}
    \clip (-1,-1) rectangle (0.26,0);
    \draw[verde](0,0) circle (1);\end{scope}
    \begin{scope}
    \clip (-1,-0.57) rectangle (-0.82,0);
    \draw[verde,thick](0,0) circle (1);
    \end{scope}
    \begin{scope}
    \clip (0.26,-1) rectangle (1,1);
    \draw[blue](0,0) circle (1);
    \end{scope}
    \begin{scope}
    \clip (-1,0) rectangle (1,1);
    \draw[blue](0,0) circle (1);
    \end{scope}
    \begin{scope}
    \clip (-1,0) rectangle (-0.09,0.99);
    \draw[blue,thick](0,0) circle (1);
    \end{scope}
    \vertex{-0.09,0.99}
    \vertex{-1,0} \vertex{0.26,-0.97}    \vertex{-0.82,-0.57}
    \end{tikzpicture}\qquad\rightarrow\qquad
    \begin{tikzpicture}
    \draw[red](-1,0)--(0,0);
    \draw[blue](1.5,0.5)--(0,0);
    \draw[verde](0.5,-0.3)--(0,0);
    \vertex{-1,0}
    \vertex{0,0}\vertex{1.5,0.5}\vertex{0.5,-0.3}
    \end{tikzpicture}    
\end{tikzcd}\caption{A tropical modification of a circle of loops of genus $6$ with a non-degenerate harmonic morphism of degree $3$ to a tree.}\label{fg:circle_loops3}
\end{figure}

Let us notice that we made no assumption on the edges' lengths, hence we can repeat the construction for any loop, which will yield a different non-degenerate harmonic morphism of degree $3$ to a tree. 
\end{ex}

In the above example, if we replace loops with different graphs, the analogue construction will not work: if instead of loops we consider hyperelliptic subgraphs with a non-trivial degree $2$ non-degenerate harmonic morphism to a tree, the resulting morphism would be degenerate.

We conclude the section by showing some properties for divisorially trigonal necklaces with $D$-hyperelliptic halves and by providing few exceptions of divisorially trigonal necklaces which are also trigonal.

\begin{lem}\label{lm:supp_neck}
    Let $\Gamma$ be a divisorially trigonal necklace. Let us denote by $\gamma$ a cycle with separating vertices. Then $D\in W_3^1(\Gamma)$ is such that $D\sim 2x+x'$ with $x,x'\in \gamma.$
\end{lem}

\begin{proof}
    We have already seen from Example \ref{ex:neck} that, when a divisor $D\in W_3^1(\Gamma)$ is supported in the cycle, then we can always write it as $D\sim 2x+x'.$ Therefore we only need to prove that the support is entirely contained in the cycle $\gamma.$

    Since $D$ has rank $1,$ we may always assume $D\sim x_1+y_1+y_2$ with $x_1\in\gamma$ and $x_1$ a separating vertex. Let us also denote by $\Gamma_1$ the connected components of $\Gamma\setminus\gamma$ glued at $x_1.$
    Assume that $y_1,y_2\not\in\gamma.$
    Then $y_1, y_2$ cannot be contained in two distinct components $\Gamma_i\neq\Gamma_1:$ any such components is connected therefore starting Dhar's burning algorithm from any of the two component at a point distinct from $y_1,y_2$ would burn the whole graph. 
    Then $y_1,y_2\in \Gamma_i$ for some $i$ and without loss of generality $\Gamma_i$ is not a necklace and $y_1,y_2$ are not in the interior of the same edge.

    Starting Dhar's burning algorithm from any point in the cycle, distinct from $x_1,$ would burn $x_1,$ and the whole cycle, but it must stop at $y_1,y_2.$ The last two burnt edges form a $2$-edge cut in $\Gamma_i.$ We can consider the rational function with slope $-1$ from $y_1,y_2$ along such a $2$-edge cut until the first reaches a vertex, towards the burnt direction. Notice that both must reach a vertex otherwise the graph would burn. We repeat this argument to see that in particular we require $y_1+y_2\sim 2w_i$ with $w_i$ the separating vertex over which the component $\Gamma_i$ is glued on the vertex.
 \end{proof}

\begin{lem}
    Let $\Gamma$ be a necklace with $D\in W_3^1(\Gamma)$. Let $\Gamma_i\subset \Gamma$ be one of the components of $\Gamma$ with a cycle removed. If $\Gamma_i$ is hyperelliptic, then it is (contained in) a $D$-hyperelliptic half.
\end{lem}

\begin{proof}
    From Lemma \ref{lm:supp_neck}, we have $D\sim 2x+x'$ with $x$ the separating vertex over which $\Gamma_i$ is glued to the cycle. 
    Assume by contradiction that there exists $H\in W_2^1(\Gamma_i)$ such that $H\not\sim 2x.$ In particular, by the uniqueness of the divisor in a hyperelliptic grapf proved in \cite[Theorem 3.6]{L} we have that $2x\not\in W_2^1(\Gamma_i).$ In other words, there exists $z,y_1,y_2\in \Gamma_i$ such that $y_1+y_2\sim 2x$ and the graph burns when we consider the divisor $y_1+y_2$ and we start a fire at $z.$ But then such a fire would spread along the rest of the graph and reach also $x'$ from the two directions of the cycle. The divisor $2x+x'-z\sim y_1+y_2+x'-z$ wouldn't then be linearly equivalent to an effective divisor, giving a contradiction.
\end{proof}

Then, if $\Gamma$ is divisorially trigonal necklace of non-loop hyperelliptic subgraphs, and thus $D$-hyperelliptic halves for some $D\in W_3^1(\Gamma)$, it might not be possible to construct a non-degenerate harmonic morphism of degree $3$ to a tree: properties (H1) and (H2) might not be satisfied.

Nonetheless, here are exceptions for which a non-degenerate harmonic morphism of degree $3$ to a tree or a tree of triangles can still be constructed, up to tropical modifications.

\begin{ex}
    Let $\Gamma$ be a divisorially trigonal necklace such that $\Gamma=\gamma\cup\Gamma_0\cup\bigcup_{i=1}^n\Gamma_i$ with $\Gamma_0$ non-loop hyperelliptic and $\Gamma_i$ are loops.
    Then there exists a non-degenerate harmonic morphism of degree $3$ from a tropical modification of $\Gamma$ to a tree. To get such a morphism it sufficed to repeat the construction in Example \ref{ex:circle_loops}, replacing the harmonic degree $2$ morphism over the non contracted with the harmonic degree $2$ morphism to a tree, determined by $\Gamma_0.$
\end{ex}

\begin{ex}
    Let $\Gamma$ be a divisorially trigonal necklace such that $\Gamma=\gamma\cup\bigcup_{i=1}^3\Gamma_i$ with $\Gamma_i$ are non-loops and hyperelliptic.
    Let $x_i$ denote the separating vertex $\gamma\cap\Gamma_i.$ 
    If $d(x_i,x_j)$ are all equal for $i,j\in\{1,2,3\}$ and $i\neq j,$
    then there exists a non-degenerate harmonic morphism of degree $3$ from a tropical modification of $\Gamma$ to a tree of triangles, as in Example \ref{ex:Luo}.
\end{ex}

\begin{rk}\label{rk:g6}
    Let us notice that, when considering a divisorially trigonal necklace, with arbitrary edge length on the cycle of separating vertices, a suitable morphism might not be constructed if we have at least $2$ hyperelliptic components which are non-loops, thus if the genus of the graph is bigger or equal than $6$.
\end{rk}

\section{Non $D$-half hyperelliptic graphs}\label{sc:non_hyp}
We would like now extend the construction of the morphism to the complement of a $D$-hyperelliptic half in $\Gamma$ with $D\in W_3^1(\Gamma)$ or more generally to $\Gamma$ 
with no $D$-hyperelliptic halves.
Once more, we will assume first that $\Gamma$ is $2$-edge connected and then generalize the construction to the case with bridges.

Let us recall the definition of \emph{admissible representative} from \cite[Definition 17]{MZ25}.
Given $D\in W_3^1(\Gamma),$ an admissible representative $D_x\sim D$ for $D$ with respect to $x\in V(G_{0})$ is $D_x=x+x_1+x_2$ such that $x_1,x_2$ do not belong to the interior of the same edge.

In the $3$-edge connected case, given $D,$ we considered all admissible representative and constructed the non-degenerate harmonic  morphism sending all points in the support of an admissible representative to the same vertex. This morphism was well-defined because, for $3$-edge connected graphs, we have proved that distinct admissible representatives have disjoint supports.

However, unlike the $3$-edge connected case, such property in general does not hold: given two distinct admissible representatived for $D,$ $D_x\sim D_y,$ it is not always true that $\operatorname{Supp}(D_x)\cap\operatorname{Supp}(D_x)=\emptyset.$ 
See for instance the example in Figure \ref{fig:adm_2_conn}.
\begin{figure}[ht]
\begin{tikzcd}
\begin{tikzpicture}
        \path[draw] (0.75,0.5) -- (2.25, 0.5);
        \draw (0.75,0.5) node[anchor=south] {$2$};
        \draw (0,-0.5) node[anchor=north] {$1$};
        \path[draw] (0,-0.5) -- (3,-0.5);
    	\begin{scope}
		\clip (0,-0.5) rectangle ++(0.75,1);
		\draw (0,0.5) ellipse (0.75 and 1);
        \draw (0.75,-0.5) ellipse (0.75 and 1);
	  \end{scope}
   \begin{scope}
		\clip (2.25,-0.5) rectangle ++(0.75,1);
		\draw (2.25,-0.5) ellipse (0.75 and 1);
        \draw (3,0.5) ellipse (0.75 and 1);
	  \end{scope}
   \divisor{0,-0.5}
   \divisor{0.75,0.5}
   \vertex{2.25,0.5}
   \vertex{3,-0.5}
        \end{tikzpicture}\qquad\sim&
    \begin{tikzpicture}
        \path[draw] (0.75,0.5) -- (2.25, 0.5);
        \draw (1.5,-0.5) node[anchor=north] {$1$};
        \draw (0.75,0.5) node[anchor=south] {$1$};
        \draw (2.25,0.5) node[anchor=south] {$1$};
        \path[draw] (0,-0.5) -- (3,-0.5);
    	\begin{scope}
		\clip (0,-0.5) rectangle ++(0.75,1);
		\draw (0,0.5) ellipse (0.75 and 1);
        \draw (0.75,-0.5) ellipse (0.75 and 1);
	  \end{scope}
   \begin{scope}
		\clip (2.25,-0.5) rectangle ++(0.75,1);
		\draw (2.25,-0.5) ellipse (0.75 and 1);
        \draw (3,0.5) ellipse (0.75 and 1);
	  \end{scope}
   \vertex{0,-0.5}
   \divisor{1.5,-0.5}
   \divisor{0.75,0.5}
   \divisor{2.25,0.5}
   \vertex{3,-0.5}
        \end{tikzpicture}\qquad\sim&
    \begin{tikzpicture}      
        \path[draw] (0.75,0.5) -- (2.25, 0.5);
        \draw (3,-0.5) node[anchor=north] {$1$};
        \draw (2.25,0.5) node[anchor=south] {$2$};
        \path[draw] (0,-0.5) -- (3,-0.5);
    	\begin{scope}
		\clip (0,-0.5) rectangle ++(0.75,1);
		\draw (0,0.5) ellipse (0.75 and 1);
        \draw (0.75,-0.5) ellipse (0.75 and 1);
	  \end{scope}
   \begin{scope}
		\clip (2.25,-0.5) rectangle ++(0.75,1);
		\draw (2.25,-0.5) ellipse (0.75 and 1);
        \draw (3,0.5) ellipse (0.75 and 1);
	  \end{scope}
   \vertex{0,-0.5}
   \vertex{0.75,0.5}
   \divisor{2.25,0.5}
   \divisor{3,-0.5}
        \end{tikzpicture}\end{tikzcd}
        
        \caption{Three admissible divisors for $D\in W_3^1(\Gamma)$, whose supports are not disjoint.}\label{fig:adm_2_conn}
\end{figure}
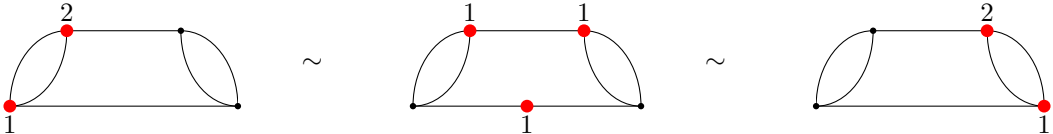

However, given a class of distinct admissible representatives with respect to a given vertex, we will choose the one having maximal coefficients with respect to that vertex.

\begin{defin}
    Let $D\in W_3^1(\Gamma).$
    Denote by $S_x\subset A_D$ the set of admissible representatives for $D$ with respect to $x\in V(G_0).$ 
    We define a maximal admissible divisor $D_x^{\operatorname{max}}$ with respect to $x$ as a divisor in $S_x$ such that $D_x^{\operatorname{max}}(x)=\operatorname{max}\{D'(x); D'\in S_x\}$.
\end{defin}

We will later prove in Lemma \ref{lm:max_unique} that maximal admissible divisors are unique.
Moreover, we will also prove that the edges between two consecutive maximal divisors are a $k$-edge cut, with $k=2,3$.

This, in general is not true, for instance if $k=2$ the support of the divisor is not always entirely contained in a $2$-edge cut. We will see that in this case then there must be a $D$-hyperelliptic half, as shown in Figure \ref{fg:counter_ex}.

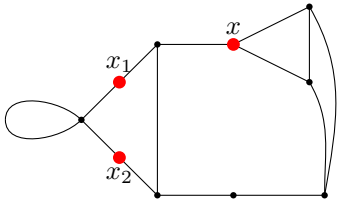
\begin{figure}[ht]
\centering
\begin{tikzcd}
    \begin{tikzpicture}
    \draw[](0,0) --  (1,1);
    \draw[](0,0) --  (1,-1);
    \draw[] (-1,0) to [out=90, in=135] (0,0);
    \draw[] (-1,0) to [out=270, in=225] (0,0);
    \draw(1,1) -- (1,-1); 
    \vertex{0,0}
\draw (0.5,0.5) node[anchor=south] {$x_1$};
\divisor{0.5,0.5}
\divisor{0.5,-0.5}
\draw (0.5,-0.5) node[anchor=north] {$x_2$};
\draw (1,1)--(2,1);
\draw (1,-1)--(2,-1);
\foreach \i in {1,2} {
\foreach \j in {1,-1}{
\vertex{\i,\j}}
}

\draw (2,1) node[anchor=south] {$x$};   
\draw (2,1)--(3,1.5);
\draw (2,1)--(3,0.5);
\draw (2,-1)--(3.2,-1);
\vertex{3,1.5}
\vertex{3,0.5}
\vertex{3.2,-1}
\divisor{2,1}
\draw(3,1.5)--(3,0.5);
\draw(3,0.5)to [out=300, in=90](3.2,-1);
\draw(3,1.5) to [out=300,in=80] (3.2,-1);
    \end{tikzpicture}
\end{tikzcd} \caption{A metric graph $\Gamma$ with a $D$-hyperelliptic half; $D\in W_3^1(\Gamma)$ whose support is not entirely contained in the $2$-cut defined by the edges containing $x_1,x_2$.}\label{fg:counter_ex}
\end{figure}

Moreover, even if the support of the divisor was entirely contained in a $2$-edge cut, it might still not be true that  the edges between two consecutive maximal divisors are a $k$-edge cut, with $k=2,3.$ See for instance Example \ref{ex:Luo}: the divisors $3p_i$, with $p_i$ the vertices of the cycle, there is a unique edge connecting the supports of the two maximal divisors $3p_i$ and $3p_j,$ with $i\neq j,$ which is not an edge-cut.

We will see that this issue arises when we are considering necklaces.
This case will be considered separately in Subsection \ref{ssc:trig_necklace}. In contrast with the case with $D$-hyperelliptic halves, we will show that when there are no hyperelliptic subgraphs, then one can construct a non-degenerate harmonic morphism of degree $3$ to a triangle of trees.

First, let us prove the following result, which also holds for necklaces.
\begin{lem}\label{lm:supp_sep_vert}
    Let $\Gamma$ be a $2$-edge connected metric graph with $D\in W_3^1(\Gamma).$
    If $p\in\Gamma$ is a separating vertex and $D\sim 2p+p'$ with $p'\neq p,$ then $\Gamma$ contains a $D$-hyperelliptic half, which moreover contains the connected component of $\Gamma \setminus \{p\}$ not containing $p'.$
\end{lem}

\begin{proof}
    Since $p$ is a separating vertex, then there are two $2$-edge connected subcurves $\Gamma_1,\Gamma_2$ whose intersection is $p.$ Assume that $p'\in\Gamma_2$, as in Figure \ref{fg:sep_vert}.
    
\begin{figure}[ht]
\centering
\begin{tikzcd}
    \begin{tikzpicture}
\draw (1.5,0) circle (0.5);
\draw (0,0) circle (1);
\draw (0,0) node[] {$\Gamma_2$};
\draw (1.5,0) node[] {$\Gamma_1$};
\draw (1,0) node[anchor=east] {$2p$};
\draw (-0.5,0.1) node[] {$p'$};
\divisor{1,0}
\divisor{-0.5,0.5}
    \end{tikzpicture}
\end{tikzcd} \caption{The graph $\Gamma$ with $D=2p+p'\in W_3^{1}(\Gamma).$}\label{fg:sep_vert}
\end{figure}
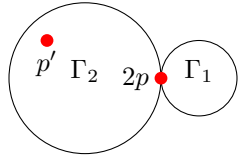

The rank of $D\sim 2p+p'$ is $1,$ therefore for any $w\in \Gamma,$ there exists $u,v\in\Gamma$ and a rational function $f$ such that
$$\operatorname{div}(f)=2p+p'-w-u-v.$$

Pick $w\in \Gamma_1;$ $,w\neq p$ then we can consider the restriction of $f$ over $\Gamma_1,\Gamma_2$ which defines principal divisors

\begin{equation}\label{eq:prin1}
\operatorname{div}(f|_{\Gamma_1})=(2-\alpha) p-w-\delta_1u-\delta_2v,
\end{equation}
\begin{equation}\label{eq:prin2}
\operatorname{div}(f|_{\Gamma_2})=\alpha p+p'+(\delta_1-1)u+(\delta_2-1)v,
\end{equation}
for some $\alpha\in\mathbb Z,$ $\delta_i\in\{0,1\}$ such that 
$
1-\alpha=\delta_1+\delta_2\in \left[0,2\right].
$

Then we have $\alpha\in\{-1,0,1\}.$
For such values of $\alpha,$ the principal divisors \eqref{eq:prin1}, \eqref{eq:prin2} give respectively linear equivalences 
\begin{equation*}
    \begin{cases}
        3p\sim w+u+v \text{ in }\Gamma_1,\\
        p\sim p' \text{ in }\Gamma_2,
    \end{cases}\qquad\text{or}\qquad
    \begin{cases}
        2p\sim w+u \text{ in }\Gamma_1,\\
        p'\sim v \text{ in }\Gamma_2,
    \end{cases}
    \qquad\text{or}\qquad
    \begin{cases}
        p\sim w \text{ in }\Gamma_1,\\
        p+p'\sim u+v \text{ in }\Gamma_2,
    \end{cases}
\end{equation*}
The first and last cases are not possible since $p'\neq p$ and $w\neq p$ by assumption and $\Gamma_1,\Gamma_2$ are $2$-edge connected. Then the only possibility is that $v=p'$ and for any $w\in \Gamma_1,$ we have that there exists some $u\in\Gamma_1$ such that $2p-w\sim u,$ which proves that $\Gamma_1$ will be contained in some $D$-hyperelliptic half.
\end{proof}

From now on, $\Gamma$ is a $2$-edge connected metric graph, which is not a necklace. Let $D=x+y+z\in W_3^1(\Gamma)$ such that $x,y$ are contained in a $2$-edge cut $\mathcal E.$ 
We want to prove that $z\in \overline{\mathcal E},$ unless there is a $D$-hyperelliptic half. 

\begin{prop}\label{prop:2cut}
Let $\Gamma$ with $D=x+y+z\in W_3^1(\Gamma)$ such that $x,y$ are contained in $e_1,e_2$ respectively, with $\mathcal E=\{e_1,e_2\}$ a $2$-edge cut. If $z\notin\overline{\mathcal E},$ then $\Gamma$ contains a $D$-hyperelliptic half.
\end{prop}

\begin{proof}
Let $\Gamma_1,\Gamma_2$ be the connected component of $\Gamma\setminus \mathcal E,$ with $z\in \Gamma_2,$ as in Figure \ref{fg:cut}.  

\begin{figure}[ht]
\centering
\begin{tikzcd}
\begin{tikzpicture}
\vertex{2.2,-0.5}
\draw (2.2,-0.5) node[anchor=south] {$z$};
\draw[](0,0) --  (0.5,0);
\draw[](1,0) --  (0.5,0);
\draw (0.25,0) node[anchor=south west] {$x$};
\draw[](-0.5,-2) -- (0.5,-2);
\draw[] (1.5,-2) -- (0.5,-2);
\draw (0,-2) node[anchor=north] {$y$};

\vertex{0.25,0}\vertex{0,-2}
\draw (0,0) to [out=260, in=45] (-0.25,-1);
\draw [](-0.25,-1) to [out=225, in=80] (-0.5,-2);
\draw [](0,0) to [out=80, in=360] (-0.25,0.25);
\draw [](-0.25,0.25) to [out=180, in=120] (-1,-0.5);
\draw [](-1,-0.5) to [out=300, in=45] (-1.5,-1.5);
\draw [](-1.5,-1.5) to [out=225, in=180] (-0.7,-2.2);
\draw [](-0.7,-2.2) to [out=360, in=260] (-0.5,-2.);
\draw (-0.6,-1) node[] {$\Gamma_1$};
\draw [](1,0) to [out=100, in=170] (2,0.2);
\draw[](2,0.2) to [out=350, in=120] (3,-0.5);
\draw [](3,-0.5) to [out=300, in=80] (2.8,-1.8);
\draw [](2.8,-1.8) to [out=260, in=10] (2.5,-2.2);
\draw [](2.5,-2.2) to [out=190, in=280] (1.5,-2);
\draw [](1.5,-2) to [out=100, in=250] (1.7,-1);
\draw [](1.7,-1) to [out=70, in=280] (1,0);
\draw (2.2,-1) node[] {$\Gamma_2$};
\end{tikzpicture}\end{tikzcd}\caption{}\label{fg:cut}
\end{figure}

First, let us suppose that in $\Gamma_1\cup\{e_1,e_2\}$ there is a unique path connecting $x$ and $y$ and that along that path all vertices are separating (notice that this also includes the case where $e_1,e_2$ share an endpoint in $\Gamma_1$).
Indeed, if $\Gamma_1$ is as in Figure \ref{fg:2cut}, we can write $D\sim x_0+y+z,$ with $x_0\in V(G_0)$ a separating vertex.

 \begin{figure}[ht]
\centering
\begin{tikzcd}
\begin{tikzpicture}
\vertex{2.2,-0.5}
\draw (2.2,-0.5) node[anchor=south] {$z$};
\draw[](0,0) --  (0.5,0);
\draw[](1,0) --  (0.5,0);
\draw (0,0) node[anchor=north west] {$x_0$};
\draw[](-0.5,-2) -- (0.5,-2);
\draw[] (1.5,-2) -- (0.5,-2);
\vertex{0,0}
\draw [](0,0) to [out=190, in=90] (-1,-1);
\draw [](-1,-1) to [out=270, in=170] (-0.5,-2);
\draw [](-0.15,0.35) circle (0.37);
\draw [](-1.5,-1) circle (0.5);
\draw (-0.5,-2.3) circle (0.3);
\draw [](1,0) to [out=100, in=170] (2,0.2);
\draw[](2,0.2) to [out=350, in=120] (3,-0.5);
\draw [](3,-0.5) to [out=300, in=80] (2.8,-1.8);
\draw [](2.8,-1.8) to [out=260, in=10] (2.5,-2.2);
\draw [](2.5,-2.2) to [out=190, in=280] (1.5,-2);
\draw [](1.5,-2) to [out=100, in=250] (1.7,-1);
\draw [](1.7,-1) to [out=70, in=280] (1,0);
\draw (2.2,-1) node[] {$\Gamma_2$};
\draw (0,-2) node[anchor=north] {$y$};
\vertex{0,-2}
\end{tikzpicture}
\end{tikzcd}\caption{} \label{fg:2cut}
\end{figure}

Starting Dhar's burning algorithm from the component glued at $x_0$ would then burn the whole graph unless either $y=x_0$ or the fire stops at $y,z$, by $2$-edge connectivity of $\Gamma$. 
If $y=x_0$, there must be a $D$-hyperelliptic half by Lemma \ref{lm:supp_sep_vert}.
If instead the fire stops at 
$y,z,$ then there is an edge $e$ containing $z$ or $e\in E_z(G_0)$ such that $\{e_2,e\}$ is a $2$-edge cut, as in Figure \ref{fg:2cut2}. Then we could consider the rational function with slope $-1$ from $y,z,$ on $e_2,e,$ respectively, towards the burnt direction until the first point reaches a vertex of the canonical model. We can repeat this argument and obtain linear equivalences $x_0+y+z\sim x_0+y_i+z_i$ where $y_i\in V(G_0)$ or $z_i\in V(G_0),$ and $y_i$ are contained in the unique path from $x_0$ to $y$, and $z_i$ are contained in a path in $\Gamma_2\cup{e_1}$ from $x_0$ to $z.$ Either $D\sim 2x_0+z'$ for some $z'\neq x_0$ or the path from $x_0$ to $z$ is also unique. This determines a unique path from $z$ to $y,$ intersecting $\Gamma_1,$ such that along that path all vertices are separating. Then $D\sim 2p+p'$ for some $p\neq p'$ and $p$ a separating vertex, hence Lemma \ref{lm:supp_sep_vert} applies again and there is a $D$-hyperelliptic half.

Let us now assume that in $\Gamma_1\cup\{e_1,e_2\}$ there is no unique path connecting $x,y$ such that along such path all vertices are separating.
Without loss of generality then we may assume $D\sim x_0+y+z$ with $x_0\in V(G_0)$ not a separating vertex. Pick $w\in V(G_0)$ a point on a path in $\Gamma_1$ between $x_0,y$ and start Dhar's burning algorithm from there. 
Since $D$ has rank $1$ the graph cannot burn entirely: either the fire stops at $x_0,y$ or it burns one vertex of the canonical model, say $x_0,$ and then stops at $y,z$.
In the following, we show that we can always write $D\sim x_0+y_0+z$ with $y_0 \in V(G_0),$ and that there exists $w'\in \Gamma_1$ such that if we start a fire from there, it stops at $x_0,y_0,$ thus identifying a new $2$-edge cut.

\begin{itemize}
    \item If the fire stops at $x_0,y,$  
    since $x_0$ is not a separating vertex and the fire does not burn $x_0$ there must be a unique path from $x_0$ to $y$ containing $w$ and the only possibility is that $y\in V(G_0).$

    \item If the fire instead burns $x_0,$ then it must stops at $y,z$.
    In particular $y,z$ identify a $2$-edge cut $\{e_2,e_z\}$
    as in Figure \ref{fg:2cut2}. We consider the rational function with slope $-1$ from $y,z,$ on $e_2,e_z,$ respectively, towards the burnt direction until the first point reaches a vertex, which is not a separating vertex along the corresponding path.
    \begin{figure}[ht]
\centering
\begin{tikzcd}
\begin{tikzpicture}
\vertex{2.2,-0.5}
\draw(2,-0.3)--(2.4,-0.7);
\draw (2.2,-0.5) node[anchor=north east] {$z$};
\draw[](0,0) --  (0.5,0);
\draw[](1,0) --  (0.5,0);
\draw (0,0) node[anchor=south west] {$x_0$};
\draw[](-0.5,-2) -- (0.5,-2);
\draw[] (1.5,-2) -- (0.5,-2);
\draw (0,-2) node[anchor=north] {$y$};

\vertex{0,0}\vertex{0,-2}
\draw (0,0) to [out=260, in=45] (-0.25,-1);
\draw [](-0.25,-1) to [out=225, in=80] (-0.5,-2);
\draw [](0,0) to [out=80, in=360] (-0.25,0.25);
\draw [](-0.25,0.25) to [out=180, in=120] (-1,-0.5);
\draw [](-1,-0.5) to [out=300, in=45] (-1.5,-1.5);
\draw [](-1.5,-1.5) to [out=225, in=180] (-0.7,-2.2);
\draw [](-0.7,-2.2) to [out=360, in=260] (-0.5,-2.);
\draw (-0.6,-1) node[] {$\Gamma_1$};
\draw [](1,0) to [out=100, in=160] (2,0.2);
\draw[](2,0.2) to [out=340, in=20] (2,-0.3);
\draw [](2,-0.3) to [out=200, in=280] (1,0);
\draw [](2.4,-0.7) to [out=30, in=80] (2.8,-1.8);
\draw [](2.8,-1.8) to [out=260, in=10] (2.5,-2.2);
\draw [](2.5,-2.2) to [out=190, in=280] (1.5,-2);
\draw [](1.5,-2) to [out=100, in=210] (2.4,-0.7);
\end{tikzpicture}&\begin{tikzpicture}
\vertex{2,-0.3}
\draw(2,-0.3)--(2.4,-0.7);
\draw (2,-0.3) node[anchor=north east] {$z_0$};
\draw (1.6,-0.1) node[] {$\Gamma_0$};
\draw[](0,0) --  (0.5,0);
\draw[](1,0) --  (0.5,0);
\draw (0,0) node[anchor=south west] {$x_0$};
\draw[](-0.5,-2) -- (0.5,-2);
\draw[] (1.5,-2) -- (0.5,-2);
\draw (-0.1,-2) node[anchor=south] {$y'$};
\draw (-0.5,-2) node[anchor=south east] {$y_0$};
\vertex{0,0}\vertex{-0.1,-2}\vertex{-0.5,-2}
\draw (0,0) to [out=260, in=45] (-0.25,-1);
\draw [](-0.25,-1) to [out=225, in=80] (-0.5,-2);
\draw [](0,0) to [out=80, in=360] (-0.25,0.25);
\draw [](-0.25,0.25) to [out=180, in=120] (-1,-0.5);
\draw [](-1,-0.5) to [out=300, in=45] (-1.5,-1.5);
\draw [](-1.5,-1.5) to [out=225, in=180] (-0.7,-2.2);
\draw [](-0.7,-2.2) to [out=360, in=260] (-0.5,-2.);
\draw (-0.6,-1) node[] {$\Gamma_1$};
\draw [](1,0) to [out=100, in=160] (2,0.2);
\draw[](2,0.2) to [out=340, in=20] (2,-0.3);
\draw [](2,-0.3) to [out=200, in=280] (1,0);
\draw [](2.4,-0.7) to [out=30, in=80] (2.8,-1.8);
\draw [](2.8,-1.8) to [out=260, in=10] (2.5,-2.2);
\draw [](2.5,-2.2) to [out=190, in=280] (1.5,-2);
\draw [](1.5,-2) to [out=100, in=210] (2.4,-0.7);
\end{tikzpicture}
\end{tikzcd}\caption{} \label{fg:2cut2}
\end{figure}

This yields $D\sim x_0+y_0+z',$ with $y_0\in V(G_0)$ not a separating vertex on a path. 
Indeed, if instead $x_0+y+z\sim x_0+y'+z_0$ with $z_0\in V(G_0),$ and $y'\notin V(G_0),$ as in the right picture in Figure \ref{fg:2cut2}, then $z_0$ must be a separating vertex and in particular there is a bridge in $\Gamma_0,$ incident to $z_0,$ over which the fire stops, otherwise $z_0$ and the whole graph would burn. If $z_0$ is a separating vertex on a path, then we could repeat the argument and get $x_0+y+z\sim x_0+y_0+z'$. Notice indeed that $x_0+y+z\not\sim 2x_0+y'$ by assumption. 

We have thus determined a linear equivalence with $D\sim x_0+y_0+z',$ where $x_0,y_0$ are now both vertices such that any edge in $\Gamma_1$ incident to them is not a bridge.
In particular starting Dhar's burning algorithm $\Gamma_1$ at some $w'$ in some path connecting $x_0,y_0$ will produce a fire which must stops at $x_0,y_0$ and thus identifying a $2$-edge cut.
\end{itemize}

Notice that in both cases $x_0,y_0$ are connected in $\Gamma_1$ via paths which are not bridges. Thus any edge in $\Gamma_1$ cannot form a $2$-edge cut with the edge identified by $z'.$

Repeating the above argument on the new $2$-edge cut, will then yield $x_0+y_0\sim x_1+y_1$ for some $x_1,y_1\in V(G_0)$ (if there are any other vertices in $\Gamma_1$). In other words, $\Gamma_1$ contains pairs of edges of same length, forming a $2$-edge cut and supporting two points of the divisor, hence it contains hyperelliptic subgraph and thus identifies a $D$-hyperelliptic half.
\end{proof}

\begin{prop}\label{prop:k_cut}
Let $\Gamma$ be a $2$-edge connected metric with $D\in W_3^1(\Gamma)$ and no $D$-hyperelliptic halves. Assume that its canonical model $(G_0,l_0)$ is such that $|V(G_0)|>3$ and let $D_x=x+x_1+x_2$ be an admissible representative for $D.$ 

Then there exists $e_{x}\in E_x(G_0),$ and $e_1,e_2\in E(G_0)$ such that $x_i\in e_i;$ $i=1,2$ and $\{e_x,e_1,e_2\},$ is a $k$-cut; $k=2,3.$ 
\end{prop}
\begin{proof}
By definition of admissible representative, we have $x\in V(G_0)$ and $x_i$ not in the interior of the same edge. 
Since $|V(G_0)|>3$ there is a vertex $w$ which is adjacent to $x$ such that the edge $xw$ does not contain $x_1,x_2.$ If we start Dhar's burning algorithm from $w,$ then there is at most one edge path to each of $x_1,x_2$ and $x$, counted with multiplicity, otherwise the graph would burn all.
The number of the disjoint paths over which the fire stops is $k$ and $\mathcal{E}$ is the set given by the last edges of such paths. 

If $k=3,$ the statement clearly follows. 
If instead  $k=2$ the claim follows from Proposition \ref{prop:2cut}.
\end{proof}

\begin{lem}\label{lm:length2}
    Let $\Gamma$ be a $2$-edge connected metric graph, with canonical model $(G_0,l_0),$ which is not a necklace with $D\in W_3^1(\Gamma).$ Suppose that there are no $D$-hyperelliptic halves and the support of $D$ is contained in a $2$-edge cut $\{e_1,e_2\}$. 
    Then $D\sim 2x_0+y$ with $x_0\in V(G_0),$ $x_0\in \overline{e}_1$  $y\in \overline{e}_2$ and $2l_0(e_1)\leq l_0(e_2).$ Moreover, equality holds if $e_1,e_2$ are parallel edges and in this case $D\sim 3x_0$. 
\end{lem}
\begin{proof}
Without loss of generality we can assume $D\sim 2x+y$ with $x,y$ in distinct edges of the $2$-edge cut, such that $x\in\mathring{e}_1$ and $y\in\mathring{e}_2.$ 

Denote by $\Gamma_1, \Gamma_2$ the two connected components of $\Gamma\setminus\{e_1,e_2\}.$ Let  $p_1,p_2\in V(G_0)$ denote the endpoints of $e_1,e_2$ in $\Gamma_2$. 

Consider first the case where $p_1\neq p_2$ and assume that $2d(x,p_1)> d(y,p_2).$ 
In particular we have that $D\sim 2p+p_2,$ and $D\sim p_1+p'+p_2,$ for some $p,p'\in\mathring{e_1},$ as in Figure \ref{fg:2_end}.

\begin{figure}[ht]
\centering
\begin{tikzcd}
    \begin{tikzpicture}
\vertex{1,0}
\draw (0,0) node[anchor=south] {$2x$};
\draw (1,0) node[anchor=west] {$p_1$};
\draw (0.5,0) node[anchor=north] {$e_1$};
\draw[](0,0)to [out=20, in=160] (1,0);
    
\draw[] (1,0) to [out=100, in=170] (2,0.2);
\draw[] (2,0.2) to [out=350, in=120] (3,-0.5);
\draw[] (2.5,-1) to [out=170, in=280] (1,0);
\draw[] (2.5,-1) to [out=350, in=300] (3,-0.5);
\draw [] (-0.25,-1)to[out=340, in=200] (2.5,-1);
\draw (-0.25,-1) node[anchor=north] {$y$};
\draw (2.5,-1) node[anchor=north] {$p_2$};
\draw (1.125,-1.25) node[anchor=south] {$e_2$};
\divisor{-0.25,-1}\vertex{2.5,-1}\divisor{0,0}
    \end{tikzpicture}\qquad\qquad \begin{tikzpicture}
\vertex{0,0}\vertex{1,0}
\draw (0,0) node[anchor=south] {$2x$};
\draw (0.5,0) node[anchor=north] {$e_1$};
\draw[](0,0)to [out=20, in=160] (1,0);
    
\draw[] (1,0) to [out=100, in=170] (2,0.2);
\draw[] (2,0.2) to [out=350, in=120] (3,-0.5);
\draw[] (2.5,-1) to [out=170, in=280] (1,0);
\draw[] (2.5,-1) to [out=350, in=300] (3,-0.5);
\draw [] (-0.25,-1)to[out=340, in=200] (2.5,-1);
\draw (-0.25,-1) node[anchor=north] {$y$};
\draw (2.5,-1) node[anchor=north] {$p_2$};
\draw (1.125,-1.25) node[anchor=south] {$e_2$};
\draw (0.8,0.05) node[anchor=south] {$2p$};
\divisor{0.8,0.05}
\vertex{-0.25,-1}\divisor{2.5,-1}
\end{tikzpicture}\qquad\qquad\begin{tikzpicture}
\vertex{0,0}\vertex{1,0}
\draw (0,0) node[anchor=south] {$2x$};
\draw (0.5,0) node[anchor=north] {$e_1$};
\draw[](0,0)to [out=20, in=160] (1,0);
    
\draw[] (1,0) to [out=100, in=170] (2,0.2);
\draw[] (2,0.2) to [out=350, in=120] (3,-0.5);
\draw[] (2.5,-1) to [out=170, in=280] (1,0);
\draw[] (2.5,-1) to [out=350, in=300] (3,-0.5);
\draw [] (-0.25,-1)to[out=340, in=200] (2.5,-1);
\draw (-0.25,-1) node[anchor=north] {$y$};
\draw (2.5,-1) node[anchor=north] {$p_2$};
\draw (1.125,-1.25) node[anchor=south] {$e_2$};
\draw (1,0) node[anchor=west] {$p_1$};
\draw (0.6,0.07) node[anchor=south] {$p'$};
\divisor{0.6,0.07}\divisor{1,0}
\vertex{-0.25,-1}\divisor{2.5,-1}
\end{tikzpicture}
\end{tikzcd} \caption{Linear equivalent divisors if $p_1\neq p_2$ and $2d(x,p_1)> d(y,p_2).$}\label{fg:2_end}
\end{figure}
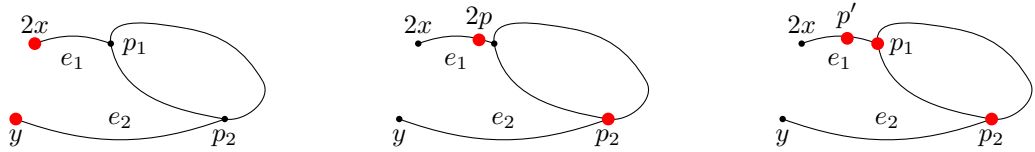

Since $\Gamma_2$ is not a $D$-hyperelliptic half, then there exists $w\in \Gamma_2$ and $q_1+q_2\sim p_1+p_2$ such that $q_1+q_2-w$ is $w$-reduced.
Then $q_1+q_2+p'\sim p_1+p_2+p'$ and $p_1+p_2+p'-w$ is also $w$-reduced (this is clear by Dhar's burning algorithm). This is not possible, therefore $2d(x,p_1)\leq d(y,p_2).$ 

Consider then the case $p_1=p_2,$ as in Figure \ref{fg:1_end}.
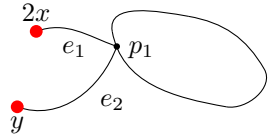
\begin{figure}[ht]
\centering
\begin{tikzcd}
    \begin{tikzpicture}
\vertex{6.06,-0.2}
\draw (5,0) node[anchor=south] {$2x$};
\draw (6.06,-0.2) node[anchor=west] {$p_1$};
\draw (5.5,0) node[anchor=north] {$e_1$};
\draw (6,-0.7) node[anchor=north] {$e_2$};
\draw[](5,0)to [out=20, in=160] (6.06,-0.2);

\draw[] (6,0) to [out=100, in=170] (7,0.2);
\draw[] (7,0.2) to [out=350, in=120] (8,-0.5);
\draw[] (7.5,-1) to [out=170, in=280] (6,0);
\draw[] (7.5,-1) to [out=350, in=300] (8,-0.5);
\draw [] (4.75,-1)to[out=340, in=250] (6.06,-0.2);
\draw (4.75,-1) node[anchor=north] {$y$};
\divisor{4.75,-1}\divisor{5,0}
    \end{tikzpicture}
\end{tikzcd} \caption{Case $p_1=p_2.$}\label{fg:1_end}
\end{figure}

The vertex $p_1$ is a separating vertex in $\Gamma$. Then by Lemma \ref{lm:supp_sep_vert} $D\sim 3p_1$ which is possible if and only if $2d(x,p_1)=d(y,p_1).$

Repeating the same argument for the component $\Gamma_1,$ then yields that $2l(e_1)\leq l(e_2)$ and equality holds if the pairs or endpoints coincide. With this inequalities between the lengths of the edges then one has that $D\sim 2x_0+y$ with $x_0\in V(G_0)$ and $D\sim 3x_0$ if $e_1,e_2$ are parallel edges. 
\end{proof}

\begin{lem}\label{lm:max_unique}Let $D_x^{\operatorname{max}}$ be an $x$-maximal admissible divisor for $D\in W_3^1(\Gamma)$, $x\in V(G_0)$.
    \begin{enumerate}
    \item If $D_x^{\operatorname{max}}(x)=1,$
    then there exist unique $x_1,x_2\neq x$ such that $D_x^{\operatorname{max}}=x+x_1+x_2.$ 
    \item If $D_x^{\operatorname{max}}(x)=2,$ then there exists a unique $x_1\neq x$ such that $D_x^{\operatorname{max}}=2x+x_1.$
    \end{enumerate}
\end{lem}
\begin{proof}
\begin{enumerate}
    \item Consider first the case $D_x^{\operatorname{max}}(x)=1,$ let $D_x^{\operatorname{max}}=x+x_1+x_2.$
    By Proposition \ref{prop:k_cut}, $x,x_1,x_2$ are supported on a $k$-edge cut, $k=2,3$. 

    If $k=2$ then by the first part of Lemma \ref{lm:length2} we would have that if $D_x^{\operatorname{max}}(x)=1$, 
    $D_x^{\operatorname{max}}=x+2y$ with $y\in V(G_0).$ If $D_x^{\operatorname{max}}=x+2y'$ with $y'\neq y$ then by \cite[Lemma 3.1. (i)]{MC} the linear equivalence $2y\sim 2y'$ would imply that there are two disjoint paths from $y$ to $y',$ not containing $x,$ forming a $2$-cut which is not possible. 
    
    Similarly, if $k=3$ and we assume by contradiction that $D_x^{\operatorname{max}}\sim x+x_1'+x_2',$ there would be linear equivalence $x_1+x_2\sim x_1'+x_2'.$ Since $\Gamma$ is $2$-edge connected then \cite[Lemma 3.1. (i)]{MC} would imply that there are two disjoint paths from $x_1,x_2$ to $x_1',x_2'$ over which the rational function defining the linear equivalence is increasing. This would only be possible if the removal of such edges would disconnect the graph giving a $2$-cut which is not possible. 

    \item For $D_x^{\operatorname{max}}(x)=2,$ let $D_x^{\operatorname{max}}=2x+y$ and by contradiction assume $D_x^{\operatorname{max}}\sim 2x+y'$ for some $y'\neq y.$ This yields $y\sim y'$ which is not possible since $\Gamma$ is bridgeless.
    \end{enumerate}
\end{proof}

Now that we have proved that maximal admissible representatives have disjoint supports, similarly to what has been done in the $3$-edge connected case, \cite[Section 3.1]{MZ25}.
Given $D\in W_3^1(\Gamma),$ we first define a refinement $(G_D,l_D)$ of the canonical model $(G_0,l_0)$ of $\Gamma,$ obtained by adding a vertex at any $y\in \operatorname{Supp}(D^{\operatorname{max}}_x),$ for any $x\in V(G_0)$ and a graph $T_D$ with a vertex $t_x$ for any maximal admissible divisor $D_x^{\operatorname{max}}\sim D;$ $x\in V(G_0).$

We construct our morphism first as a well-defined map of vertices $\varphi_D:V(G_D)\to V(T_D).$ For any maximal admissible representative $D_x^{\operatorname{max}}=x+x_1+x_2,$ set $\varphi_D(x)=\varphi_D(x_1)=\varphi_D(x_2)=t_x.$

Then, in order to extend the morphism on $E(G_D)$, we need to consider also here consecutive maximal admissible representatives, i.e, two maximal representatives whose supports define the endpoints of edges in such a refinement, and define the morphims along the edges.

\begin{prop}\label{prp:lengths}
Let $\Gamma$ be a metric graph with $D\in W_3^1(\Gamma)$ and no $D$-hyperelliptic halves.
Let $D_x^{\operatorname{max}}$ be a maximal admissible representative for $D$. Then there is a distinct admissible representative $D_y^{\operatorname{max}}$ consecutive to it.
Furthermore, the edges with endpoints contained in two consecutive maximal admissible representative form a $k$-edge cut with $k=2, 3.$ 
\end{prop}

\begin{proof}
    Let $D_x^{\operatorname{max}}$ be a maximal admissible representative for some divisor $D$. Then by Proposition \ref{prop:k_cut} there is a $k$-edge cut containing the points in its support.
    
    If $k=3$ then we can consider the rational function with slope $-1$ from each of the points (with multiplicities) in the support of $D_x^{\operatorname{max}}$ along paths of length equal to that of the shortest edge of the cut, and let $y\in V(G_0)$ denote its endpoint. This identifies the consecutive admissible divisor $D_y$.
    Assume that $D_y(y)=\alpha$ with $\alpha\leq 3$ the number of edges in the $3$-edge cut, in $G_0$ incident to $y$. If by contradiction there is an admissible divisor $D'\sim D_y$ with $D'(y)=\beta>\alpha$ , then we would have linear equivalences $y_1+y_2\sim y+y_3,$ or $y_1\sim y$ for $y_i\neq y$. This is not possible again by \cite[Lemma 3.1. (i)]{MC}.
        
    If instead $k=2,$ then from Lemma \ref{lm:length2} the consecutive admissible divisor would then be $2y+u'',$ for some $u''$ in $e_2,$ or $D_y\sim 3y$ if the two endpoints coincide,  which are maximal.
\end{proof}

\begin{rk}\label{rk:bridges}
Observe that if $\Gamma$ contains a bridge then the graph $\Gamma',$ obtained by contracting the bridge,  contains a separating vertex $v.$

Let us assume that $\Gamma$ has bridges with $D\in W_3^1(\Gamma)$ and no $D$-hyperelliptic halves. Denote by $D'$ its image via the contraction of bridges. By \cite[Corollaries 5.10,5.11]{BN} the divisor $D'$ has still rank $1$, thus we can consider $D_v\sim D'$.
By being a separating vertex, $v$ is incident to $m$ distinct $k_i$-edge cuts with $k_i\in \{2,3\}$ by Proposition \ref{prop:k_cut}.
If for any $i$, $k_i=3$, again by Dhar's burning algorithm, the graph burns all unless $v_1=v_2=v.$
    Instead, if any of the edge cuts consists of $2$ edges, then the second part of the proof of Proposition \ref{prp:lengths} proves that again $D'\sim 3v.$

    This shows that, if $\Gamma$ has a bridge $b$, then $D\sim 3p,$ for any $p\in b.$ Therefore we can set $D\sim D_u^{\operatorname{max}}=3u$ and $D\sim D_w^{\operatorname{max}}=3w,$ with $u,w$ the two endpoints of $b.$
\end{rk}

\begin{rk}\label{rk:lengths_eq}
Let us consider $(G_D,l_D)$ the refinement of $(G_0,l_0)$ obtained by inserting a vertex in the support of each maximal admissible divisor. Then the edges between two consecutive maximal divisor form a $k$-edge cut for $k=2,3.$ In particular if $k=3$ they have same length $l$ in $G,$ while if $k=2,$ and the $k$-edge cut is formed by two edges $e_1,e_2$ such that $2l_0(e_1')\leq l_0(e_2')$, where $e_i\subset e_i'\in E(G_0),$ as proved in Proposition \ref{prp:lengths}, then $2l(e_1)=l(e_2).$ 
\end{rk}

We now extend the map $\varphi_D$ on $E(G_D).$
For any pair of vertices $t_x,t_y\in V(T_D)$ we set $t_xt_y\in E(T_D)$ if and only if $D_x^{\operatorname{max}},D_y^{\operatorname{max}}$ are consecutive and we extend $\varphi_D$ to $V(G_D)\cup E(G_D)\to V(T_D)\cup E(T_D)$ such that for any $x'y'\in E(G_D)$
$$
\varphi_D(x'y')=
\begin{cases}t_xt_y\in E(T_D) &\text{if }x'\in \operatorname{Supp}(D_x^{\operatorname{max}}), y'\in \operatorname{Supp}(D_y^{\operatorname{max}})\\
t_x\in V(T_D) &\text{if }x',y'\in \operatorname{Supp}(D_x^{\operatorname{max}}).
\end{cases}$$

For any pair of consecutive maximal admissible divisors, let us consider the $k$-edge cut that they define. If $k=3,$ we assign to each edge $e$ in the $3$-edge cut the index $\mu_{\varphi_D}(e)=1$ and set $l_{T_D}(\varphi_D(e))=l_D(e).$
If $k=2,$ then by Remark \ref{rk:lengths_eq} the two edges of the cut are such that the longest $e_2$ is twice the shortest $e_1$. We set index $1$ on the longest and $2$ on the shortest and $l_{T_D}(\varphi_D(e))=l_D(e_2)=2l_D(e_1)$.
If instead $k=1,$ we assign to the bridge $b$ index $3$ and $l_{T_D}(\varphi_D(b))=3l_D(b).$ 
This yields an indexed morphism which induces a morphism on metric graphs $\varphi_D:(G_D,l_D)\to (T_D,l_{T_D}).$ 

\begin{prop}\label{prp:harmonic}
Let $\Gamma$ be a $2$-edge connected metric graph which is not a necklace, with $D\in W_3^1(\Gamma)$ and no $D$-hyperelliptic halves.
The morphism $\varphi_D:(G_D,l_D)\to(T_D,l_{T_D})$ constructed above is non-degenerate and harmonic of degree $3$ to a metric tree.
\end{prop}
\begin{proof}
Let us show first that the graph $T_D$ is a tree. By construction,
$T_D=\varphi_D(G_D)$ and it is connected since $G_D$ is: any edge $e\in E(T_D)$ is such that the edges in $\varphi_D^{-1}$ form a $k$-edge cut, for some $k\leq 3$ and
therefore the removal of $e$ must disconnect $T_D$.
The morphism is non-degenerate by Proposition \ref{prp:lengths} and clearly has degree $3$, if harmonic.

Let $x\in V(G_D)$ in the pre-image of a vertex $t$ of the tree $T_D,$ to prove that the morphism is harmonic, we have to show that the quantity
$$m_{\varphi_D}(x)=\sum_{\substack{e\in E_x(G_D)\\ \varphi_D(e)=e'}}\mu_{\varphi_D}(e)$$ is constant for any $e'\in E_t(T_D).$

Similarly to the $3$-edge connected case, because of how we defined the indices of the morphism, such quantity only depends on the coefficients of the points on the support of $D^{\operatorname{max}}_{x'},$ the unique maximal admissible divisor containing $x,$ thus harmonicity follows.
\end{proof}

Combining this with the construction over $D$-hyperelliptic halves, completes the proof of Theorem \ref{th:main_general}.

\begin{proof}[Proof of Thereom \ref{th:main_general} $B.\Rightarrow A$]
If $\Gamma$ is hyperelliptic, then Proposition \ref{prp:hyp_trig} applies.

Let $D\in W_3^1(\Gamma)$. If instead $\Gamma$ is not hyperelliptic, then we can always write $$\Gamma=\Gamma_0\sqcup \left(\bigcup_{i=1}^m\Gamma_i\right)$$ where $\Gamma_i$ are the $D$-hyperelliptic halves, for $i>0,$ where $|\Gamma_i\cap\Gamma_j|\leq1$ for $i,j>0$; $i\neq j$ by Lemma \ref{lm:inters}. 

If $\Gamma_0$ is non-empty, then we apply the construction made in this section for divisorially trigonal graphs with no $D$-hyperelliptic halves. This gives a non-degenerate harmonic morphism $\varphi_D^0$ of degree $3$ from a refinement $(G_D^0,l_D^0)$ of the canonical model of $\Gamma_0$ to a metric tree $(T_D^0,l_D^0),$ as proved in Proposition \ref{prp:harmonic}.

Now, over each $\Gamma_i$, for each $i>0$, we know that $D\sim H_i+p_i$ with $H_i\in W_2^1(\Gamma_i)$, so we can apply the construction from Section \ref{sc:hyperelliptic} to obtain a non-degenerate harmonic morphism of degree $3$ $\varphi_D^i:\Gamma_i\cup T_i\to T_i$, where $T_i$ is the target tree of the degree $2$ non-degenerate harmonic morphism associated to $H_i.$
In particular, the $i$-th tree is glued at $p_i$ and from Proposition \ref{prp:morph_hyp} we know that such morphisms glue over the non-empty intersection of distinct $D$-hyperelliptic halves to yield a harmonic non-degenerate degree $3$ morphism $\varphi_D^{hyp}$from the union of the $\Gamma_i$'s to a tree $T^{hyp}$. 

We denote by $\varphi_D:(G_D,l_D)\to (T_D,l_{T_D})$ the morphism obtained by  the gluing $\varphi_D^0$ and $\varphi_D^{hyp}$ at their intersection points, as in Figure \ref{fg:ex_morph}. 

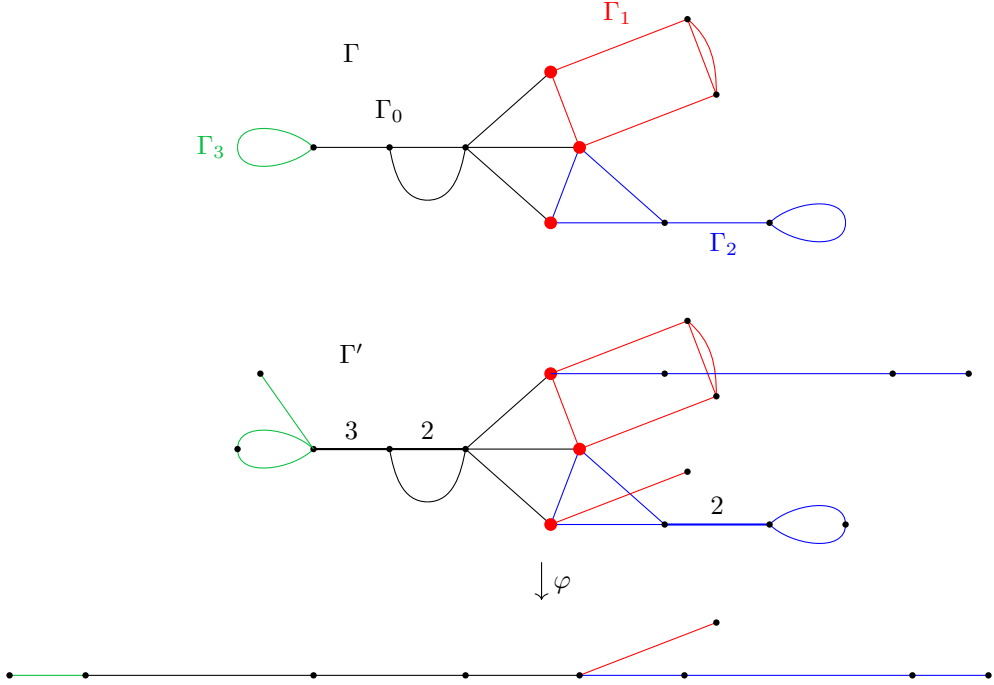
\begin{figure}[ht]
\centering
\begin{tikzcd}
\begin{tikzpicture}
\draw (0.5,5) node[anchor=south] {$\Gamma$};
\draw (-1,4)[verde] to [out=90, in=135] (0,4);
\draw (-1,4)[verde] to [out=270, in=225] (0,4);
\draw (-1,4)[verde] node[anchor=east] {$\Gamma_3$};
\draw (0,4) -- (1,4);
\draw (1,4) -- (2,4);
\draw (1,4) to [out=280, in=180] (1.5,3.3);
\draw (1.5,3.3) to [out=0, in=260] (2,4);
\draw (2,4) -- (3.12,5);
\draw (2,4) -- (3.12,3);
\draw (2,4) -- (3.5,4);
\draw[red] (3.12,5)--(3.5,4);
\draw [blue](3.12,3)--(3.5,4);
\vertex{0,4}\vertex{1,4}\vertex{2,4}
\draw[red] (3.5,4)--(5.3,4.7);
\draw[red] (3.12,5)--(4.92,5.7);
\draw[red] (5.3,4.7)--(4.92,5.7);
\draw[red] (5.3,4.7) to [out=90, in=320] (4.92,5.7);
\draw[blue](3.5,4)--(4.62,3);
\draw[blue](3.12,3)--(4.62,3);
\draw[blue](4.62,3)--(6,3);
\draw[blue] (6,3) to [out=45, in=90] (7,3);
\draw[blue] (6,3) to [out=315, in=270] (7,3);
\vertex{5.3,4.7}\vertex{4.92,5.7}\vertex{6,3}\vertex{4.62,3}
\divisor{3.5,4}\divisor{3.12,5}\divisor{3.12,3}
\draw (1,4.2) node[anchor=south] {$\Gamma_0$};
\draw[red] (4,5.5) node[anchor=south] {$\Gamma_1$};
\draw[blue] (5.4,3) node[anchor=north] {$\Gamma_2$};

\draw (0.5,1) node[anchor=south] {$\Gamma'$};
\draw (-1,0)[verde] to [out=90, in=135] (0,0);
\draw (-1,0)[verde] to [out=270, in=225] (0,0);
\draw[thick] (0,0) -- (1,0);
\draw[thick] (1,0) -- (2,0);
\draw (0.5,0) node[anchor=south] {$3$};
\draw(1.5,0) node[anchor=south] {$2$};
\draw (1,0) to [out=280, in=180] (1.5,-0.7);
\draw (1.5,-0.7) to [out=0, in=260] (2,0);
\draw (2,0) -- (3.12,1);
\draw (2,0) -- (3.12,-1);
\draw (2,0) -- (3.5,0);
\draw[red] (3.12,1)--(3.5,0);
\draw [blue](3.12,-1)--(3.5,0);
\vertex{0,0}\vertex{1,0}\vertex{2,0}
\draw[red] (3.5,0)--(5.3,0.7);
\draw[red] (3.12,1)--(4.92,1.7);
\draw[red] (5.3,0.7)--(4.92,1.7);
\draw[red] (5.3,0.7) to [out=90, in=320] (4.92,1.7);
\draw[blue](3.5,0)--(4.62,-1);
\draw[blue](3.12,-1)--(4.62,-1);
\draw[blue,thick](4.62,-1)--(6,-1);
\draw[blue] (6,-1) to [out=45, in=90] (7,-1);
\draw[blue] (6,-1) to [out=315, in=270] (7,-1);
\vertex{5.3,0.7}\vertex{4.92,1.7}\vertex{6,-1}\vertex{4.62,-1}
\divisor{3.5,0}\divisor{3.12,1}\divisor{3.12,-1}
\vertex{-1,0}
\draw[verde](0,0)--(-0.7,1);
\vertex{-0.7,1}
\draw[red] (3.12,-1)--(4.92,-0.3);
\vertex{4.92,-0.3}
\vertex{7,-1}
\draw[blue] (3.12,1)--(8.62,1);
\vertex{4.62,1}\vertex{7.62,1}\vertex{8.62,1}
\draw[->] (3,-1.5) to (3,-2);
\draw (3,-1.75) node[anchor=west] {$\varphi$};
\draw[verde](-4,-3)--(-3,-3);
\draw(-3,-3)--(3.5,-3);
\draw[red] (3.5,-3)--(5.3,-2.3);
\draw[blue] (3.5,-3)--(8.88,-3);
\vertex{-3,-3}\vertex{0,-3}\vertex{3.5,-3}\vertex{2,-3}\vertex{-4,-3}\vertex{5.3,-2.3}\vertex{4.88,-3} \vertex{7.88,-3}\vertex{8.88,-3}
\draw (5.31,-1) node[anchor=south] {$2$};
\end{tikzpicture}
\end{tikzcd}
\caption{A metric graph $\Gamma$ with $D\in W_3^1(\Gamma),$ and a non-degenerate harmonic morphism of degree $3$ from a tropical modificaton $\Gamma'$ of $\Gamma$ to a tree.}\label{fg:ex_morph}
\end{figure}

The graph $T_D$ is a tree since it has been obtained by gluing the trees and no cycle can be defined after gluing  by property (H1).
Again, no contraction has been added; therefore, $\varphi_D$ is non-degenerate. 

What is left to check is that the resulting morphism is still harmonic over the points over which a $D$-hyperelliptic half and its complement glue. This is again a consequence of the fact that the quantity $m_{\varphi_D}(x)$ only depends on the coefficients of the maximal admissible divisors, but this is also true over the $D$-hyperelliptic half and for $x=p_i$ because of the addition of the tree in the construction.
\end{proof}

\subsection{Non-hyperelliptic necklaces} \label{ssc:trig_necklace}

Let us now finally consider a metric graph $\Gamma$ which is a non-hyperelliptic necklace. We show that if $\Gamma$ is divisorially trigonal then there exists a non-degenerate harmonic morphism of degree $3$ from a tropical modification of $\Gamma$ to a tree of triangles, thus concluding the proof of Theorem \ref{th:main_neck}.

Let us recall from Lemma \ref{lm:supp_neck} that given $D\in W_3^1(\Gamma)$ we may always assume that its support is entirely contained in a cycle $\gamma$ of the necklace and that $D\sim 2x+x'$ for any $x\in\gamma.$

If we moreover assume that $\Gamma$ doesn't contain $D$-hyperelliptic halves, then we can give a further description of our divisor: by Lemma \ref{lm:supp_sep_vert}, if $\Gamma_i$ is the component of $\Gamma$ with the cycle removed, glued on the vertex $x_i$, then $D\sim 3x_i.$

We now show that if this is the case the number of component in necessarily $3,$ and the necklace is a triangle, i.e. the edges in the disconnecting cycle are precisely 3.

\begin{lem}\label{lm:3comp}
    Let $\Gamma$ be a divisorially trigonal non-hyperelliptic necklace.
    Then for any cycle $\gamma$ of separating vertices, $\Gamma\setminus \gamma$ is given by exactly $3$ components $\Gamma_1,\Gamma_2,\Gamma_3$ which are also divisorially trigonal and $d(x_1,x_2)=d(x_2,x_3)=d(x_1,x_3)$ where $x_i=\gamma\cap\Gamma_i.$ 
\end{lem}

\begin{proof}
Let $D\in W_3^1(\Gamma).$ By Lemmas \ref{lm:supp_neck}, \ref{lm:supp_sep_vert}, we have $D\sim 3x_i$ for any $x_i$ separating vertex of $\gamma.$ Clearly $3x_i\in W_3^1(\Gamma_i),$ where $\Gamma_i$ denotes the sub-curve not containing $\gamma$ glued at $x_i.$

In particular, by definition of necklace, we have at least $3$ vertices and the linear equivalence $3x_1\sim 3x_2$ shows that the cycle $\gamma$ is given by the union of two paths $P_1,P_2$ such that $l(P_i)=2l(P_j).$ 

Repeating this for the pair $x_2,x_3$ then yields that all edges whose endpoints are both separating vertices have all the same length and they cannot be more than $3.$
\end{proof}

The above lemma then proves that any divisorially trigonal non-hyperelliptic necklace must be as in Figure \ref{fg:Luo2}. 

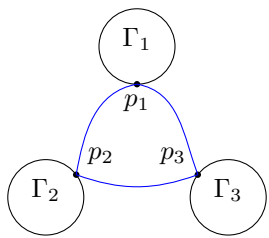
\begin{figure}[ht]\begin{tikzcd}
\begin{tikzpicture}
    \draw (0.8,0.2) circle (0.5);
    \draw (0.8,0.2) node {$\Gamma_1$};
    \draw (-0.4,-1.8) circle (0.5);
    \draw (-0.4,-1.8) node {$\Gamma_2$};
    \draw (2,-1.8) circle (0.5);
    \draw (2,-1.8) node {$\Gamma_3$}; \vertex{0,-1.5}\vertex{0.8,-0.3}\vertex{1.6,-1.5}
    \draw[blue] (0,-1.5) to [out=340, in=200] (1.6,-1.5);
    \draw[blue] (0,-1.5) to [out=80, in=190] (0.8,-0.3);
    \draw[blue] (1.6,-1.5) to [out=110, in=350] (0.8,-0.3);
    \draw (0,-1.5) node[anchor=south west] {$p_2$};
    \draw (0.8,-0.3) node[anchor=north] {$p_1$};
    \draw (1.6,-1.5) node[anchor=south east] {$p_3$};
    \end{tikzpicture}    
\end{tikzcd}\caption{A metric graph, with edges of the same color with the same length and $\Gamma_i$ divisorially trigonal.}\label{fg:Luo2}
\end{figure}

In particular the argument in \cite[Example 5.13]{ABBR} for the non-existence of a non-degenerate harmonic morphism of degree $3$ to a tree (up to tropical modification) still holds. However we now show that a morphism with the same properties to a tree of triangles can be constructed, thus proving Theorem \ref{th:main_neck}.

\begin{proof}[Proof of Theorem \ref{th:main_neck}, $B.\Rightarrow A$]
    Let $\Gamma_1\subset \Gamma$ be a component of a necklace, which is not itself a necklace. Then by Lemma \ref{lm:3comp} it is divisorially trigonal and by Theorem \ref{th:main_general} there exists a non-degenerate harmonic morphism of degree $3$ from $\Gamma_1'$ to a tree $T_1$, with $\Gamma_1'$ a tropical modification of $\Gamma_1.$ 
    
    In particular, $\Gamma_1$ is one of three components glued to a cycle, whose edges have the same lengths by Lemma \ref{lm:3comp}. 
    We can then extend the morphism by sending the cycle to a cycle over with the same number of vertices, with edge lengths multiplied by $3$, and gluing it to a leaf of the tree. 
    The morphism is again non-degenerate since no contraction has been added and it is still harmonic: $D\sim 3x_1,$ where $x_1$ is the gluing point. 

    Repeating iteratively this contruction for all non-necklaces components within any necklace then yields a non-degenarate harmonic morphism of degree $3$ and the target space is a tree of triangles.
\end{proof}

Finally, let us provide a proof of Corollary \ref{cor:main_g6}.

\begin{proof}[Proof of Corollary \ref{cor:main_g6}]

It is sufficient to prove that divisorially trigonal graphs for which a non-degenerate harmonic morphism of degree $3$ to a tree doesn't exist (even if tropical modification are allowed) have genus higher or equal than $6$.

Consider a graph as the one in Figure \ref{fg:Luo2}. If for instance $\Gamma_1$ is a loop, while $\Gamma_2,\Gamma_3$ are divisorially trigonal graph with a morphism of degree $3$ (from a tropical modification) to trees $T_2,T_3$, then we can still construct a morphism from the whole graph which is non-degenerate, harmonic and of degree $3$. Indeed, as in Figure \ref{fg:Luo_loop}, we can extend the two morphisms to the two trees by adding a vertex $v$ at the midpoint of the edge $p_2p_3$ and then send the edges $p_1p_2$, $p_2v$ to the same edge with multiplicities $1$ and $2,$ respectively. Similarly, send the edged $p_1p3$, $p_3v$ to the same edge with multiplicities $1$ and $2,$ respectively, and contract the loop.

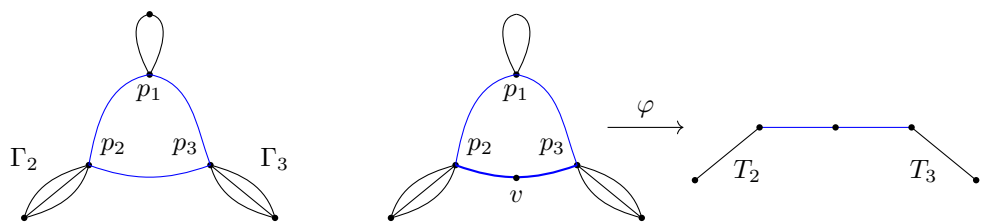
\begin{figure}[ht!]\begin{tikzcd}
\begin{tikzpicture}
    \draw(0.8,0.5)[] to [out=190, in=120] (0.8,-0.3);
    \draw(0.8,0.5)[] to [out=350, in=60] (0.8,-0.3);
    \vertex{0.8,0.5}
    \draw(-0.85,-2.2)[] to [out=75, in=195] (0,-1.5);
    \draw(-0.85,-2.2)[] to [out=15, in=255] (0,-1.5);
    \draw[](0,-1.5)--(-0.85,-2.2);\vertex{-0.85,-2.2}
    \draw(1.6,-1.5)[] to [out=285, in=165] (2.45,-2.2);
    \draw(1.6,-1.5)[] to [out=345, in=105] (2.45,-2.2);
    \draw[](1.6,-1.5)--(2.45,-2.2);\vertex{2.45,-2.2}
\vertex{0,-1.5}\vertex{0.8,-0.3}\vertex{1.6,-1.5}
    \draw[blue] (0,-1.5) to [out=340, in=200] (1.6,-1.5);
    \draw[blue] (0,-1.5) to [out=80, in=190] (0.8,-0.3);
    \draw[blue] (1.6,-1.5) to [out=110, in=350] (0.8,-0.3);
    \draw (0,-1.5) node[anchor=south west] {$p_2$};
    \draw (0.8,-0.3) node[anchor=north] {$p_1$};
    \draw (1.6,-1.5) node[anchor=south east] {$p_3$};
    \draw (-0.5,-1.7) node[anchor=south east] {$\Gamma_2$};
    \draw (2.1,-1.7) node[anchor=south west] {$\Gamma_3$};
    \end{tikzpicture} &\begin{tikzpicture}
    \draw(0.8,0.5)[] to [out=190, in=120] (0.8,-0.3);
    \draw(0.8,0.5)[] to [out=350, in=60] (0.8,-0.3);
    \draw(-0.85,-2.2)[] to [out=75, in=195] (0,-1.5);
    \draw(-0.85,-2.2)[] to [out=15, in=255] (0,-1.5);
    \draw[](0,-1.5)--(-0.85,-2.2);\vertex{-0.85,-2.2}
    \draw(1.6,-1.5)[] to [out=285, in=165] (2.45,-2.2);
    \draw(1.6,-1.5)[] to [out=345, in=105] (2.45,-2.2);
    \draw[](1.6,-1.5)--(2.45,-2.2);\vertex{2.45,-2.2}
\vertex{0,-1.5}\vertex{0.8,-0.3}\vertex{1.6,-1.5}
    \draw[thick,blue] (0,-1.5) to [out=340, in=200] (1.6,-1.5);
    \draw[blue] (0,-1.5) to [out=80, in=190] (0.8,-0.3);
    \draw[blue] (1.6,-1.5) to [out=110, in=350] (0.8,-0.3);
    \draw (0,-1.5) node[anchor=south west] {$p_2$};
    \draw (0.8,-0.3) node[anchor=north] {$p_1$};
    \draw (0.8,-1.7) node[anchor=north] {$v$};
    \vertex{0.8,-1.67}
    \draw (1.6,-1.5) node[anchor=south east] {$p_3$};
    \draw[->](2,-1)--(3,-1);
    \draw (2.5,-1) node[anchor=south] {$\varphi$};
    \draw[blue](4,-1)--(6,-1);
    \vertex{5,-1}
    \draw[](4,-1)--(3.15,-1.7);
    \draw[](6,-1)--(6.85,-1.7);
    \draw (3.5,-1.3) node[anchor=north west] {$T_2$};
    \draw (6.5,-1.3) node[anchor=north east] {$T_3$};
    \vertex{6,-1}\vertex{4,-1}\vertex{3.15,-1.7}\vertex{6.85,-1.7}
    \end{tikzpicture}   
\end{tikzcd}\caption{A metric graph, with a non-degenerate harmonic morphism of degree $3$ from a tropical modification to a metric tree.}\label{fg:Luo_loop}
\end{figure}

 A similar construction can also be done if the loop is glued via a bridge, see Figure \ref{fg:Luo_loop2}.
\begin{figure}[ht]\begin{tikzcd}
\begin{tikzpicture}
    \draw(0.8,1)[] to [out=190, in=120] (0.8,0.2);
    \draw(0.8,1)[] to [out=350, in=60] (0.8,0.2);
    \vertex{0.8,1}
    \vertex{0.8,0.2}
    \draw[](0.8,0.2)--(0.8,-0.3);
    \draw(-0.85,-2.2)[] to [out=75, in=195] (0,-1.5);
    \draw(-0.85,-2.2)[] to [out=15, in=255] (0,-1.5);
    \draw[](0,-1.5)--(-0.85,-2.2);\vertex{-0.85,-2.2}
    \draw(1.6,-1.5)[] to [out=285, in=165] (2.45,-2.2);
    \draw(1.6,-1.5)[] to [out=345, in=105] (2.45,-2.2);
    \draw[](1.6,-1.5)--(2.45,-2.2);\vertex{2.45,-2.2}    \vertex{0,-1.5}\vertex{0.8,-0.3}\vertex{1.6,-1.5}
    \draw[blue] (0,-1.5) to [out=340, in=200] (1.6,-1.5);
    \draw[blue] (0,-1.5) to [out=80, in=190] (0.8,-0.3);
    \draw[blue] (1.6,-1.5) to [out=110, in=350] (0.8,-0.3);
    \draw (0,-1.5) node[anchor=south west] {$p_2$};
    \draw (0.8,-0.3) node[anchor=north] {$p_1$};
    \draw (1.6,-1.5) node[anchor=south east] {$p_3$};
    \draw (-0.5,-1.7) node[anchor=south east] {$\Gamma_2$};
    \draw (2.1,-1.7) node[anchor=south west] {$\Gamma_3$};
    \end{tikzpicture} &\begin{tikzpicture}
        \draw(0.8,1)[] to [out=190, in=120] (0.8,0.2);
    \draw(0.8,1)[] to [out=350, in=60] (0.8,0.2);
    \vertex{0.8,0.2}
    \draw[red](0.8,0.2)--(0.8,-0.3);
    \draw(-0.85,-2.2)[] to [out=75, in=195] (0,-1.5);
    \draw(-0.85,-2.2)[] to [out=15, in=255] (0,-1.5);
    \draw[](0,-1.5)--(-0.85,-2.2);\vertex{-0.85,-2.2}
    \draw(1.6,-1.5)[] to [out=285, in=165] (2.45,-2.2);
    \draw(1.6,-1.5)[] to [out=345, in=105] (2.45,-2.2);
    \draw[](1.6,-1.5)--(2.45,-2.2);\vertex{2.45,-2.2}
\vertex{0,-1.5}\vertex{0.8,-0.3}\vertex{1.6,-1.5}
    \draw[thick,blue] (0,-1.5) to [out=340, in=200] (1.6,-1.5);
    \draw[blue] (0,-1.5) to [out=80, in=190] (0.8,-0.3);
    \draw[blue] (1.6,-1.5) to [out=110, in=350] (0.8,-0.3);
    \draw (0,-1.5) node[anchor=south west] {$p_2$};
    \draw (0.8,-0.3) node[anchor=north] {$p_1$};
    \draw (0.8,-1.7) node[anchor=north west] {$v$};
    \draw[red,thick] (0.8,-1.7) --(0.8,-1.95);
    \vertex{0.8,-1.67}
    \vertex{0.8,-1.95};
    \draw (1.6,-1.5) node[anchor=south east] {$p_3$};
    \draw[->](2,-1)--(3,-1);
    \draw (2.5,-1) node[anchor=south] {$\varphi$};
    \draw[blue](4,-1)--(6,-1);
    \draw[red](5,-1)--(5,-0.5);
    \vertex{5,-1}   \vertex{5,-0.5}
    \draw[](4,-1)--(3.15,-1.7);
    \draw[](6,-1)--(6.85,-1.7);
    \draw (3.5,-1.3) node[anchor=north west] {$T_2$};
    \draw (6.5,-1.3) node[anchor=north east] {$T_3$};
    \vertex{6,-1}\vertex{4,-1}\vertex{3.15,-1.7}\vertex{6.85,-1.7}
    \end{tikzpicture}   
\end{tikzcd}\caption{A metric graph, with a non-degenerate harmonic morphism of degree $3$ from a tropical modification to a metric tree.}\label{fg:Luo_loop2}
\end{figure}
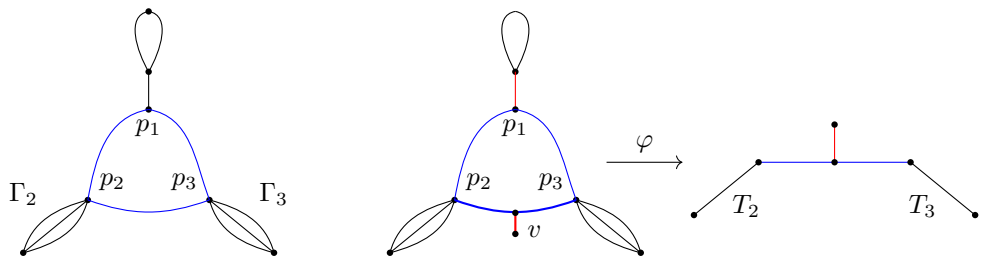

Therefore a divisorially trigonal metric graph $\Gamma$ of the form represented in Figure \ref{fg:Luo2}, with no non-degenerate harmonic morphisms of degree 3 from any of its tropical modifications to a tree, is such that each $\Gamma_i$ has genus greater than or equal to 2. Therefore if $\Gamma$ has genus greater than or equal to $7$, the above construction might be applied and a non-degenerate harmonic morphism of degree $3$ to a tree might not exist.

When we consider necklaces with hyperelliptic components, however, as shown in Remark \ref{rk:g6}, a non-degenerate harmonic morphism of degree $3$ to a tree might not exist even for genus $6$, whereas if $g\leq 5$, if $\Gamma$ is a divisorially trigonal necklace then at least two components $\Gamma_i,\Gamma_j$ are loops and the morphism must exist.
\end{proof} 

\bibliographystyle{alpha}
\bibliography{bib}

\begin{thebibliography}{ABBR15}

\bibitem[ABBR15]{ABBR}
O.~Amini, M.~Baker, E.~Brugall\'e, and J.~Rabinoff.
\newblock Lifting harmonic morphisms {II}: {T}ropical curves and metrized complexes.
\newblock {\em Algebra Number Theory}, 9(2):267--315, 2015.

\bibitem[AC13]{AC}
O.~Amini and L.~Caporaso.
\newblock Riemann-{R}och theory for weighted graphs and tropical curves.
\newblock {\em Adv. Math.}, 240:1--23, 2013.

\bibitem[ACP22]{ACP22}
D.~Allcock, D.~Corey, and S.~Payne.
\newblock Tropical moduli spaces as symmetric {$\Delta$}-complexes.
\newblock {\em Bull. Lond. Math. Soc.}, 54(1):193--205, 2022.

\bibitem[BN09]{BN}
M.~Baker and S.~Norine.
\newblock Harmonic morphisms and hyperelliptic graphs.
\newblock {\em Int. Math. Res. Not. IMRN}, 2009(15):2914--2955, 2009.

\bibitem[BS13]{BS}
M.~Baker and F.~Shokrieh.
\newblock Chip-firing games, potential theory on graphs, and spanning trees.
\newblock {\em J. Combin. Theory Ser. A}, 120(1):164--182, 2013.

\bibitem[Cap14]{LC}
L.~Caporaso.
\newblock Gonality of algebraic curves and graphs.
\newblock In {\em Algebraic and complex geometry}, volume~71 of {\em Springer Proc. Math. Stat.}, pages 77--108. Springer, Cham, 2014.

\bibitem[Cha13]{MC}
M.~Chan.
\newblock Tropical hyperelliptic curves.
\newblock {\em J. Algebraic Combin.}, 37(2):331--359, 2013.

\bibitem[CKK15]{CKK}
G.~Cornelissen, F.~Kato, and J.~Kool.
\newblock A combinatorial {L}i-{Y}au inequality and rational points on curves.
\newblock {\em Math. Ann.}, 361(1-2):211--258, 2015.

\bibitem[CMR16]{CMR}
R.~Cavalieri, H.~Markwig, and D.~Ranganathan.
\newblock Tropicalizing the space of admissible covers.
\newblock {\em Math. Ann.}, 364(3-4):1275--1313, 2016.

\bibitem[Dha90]{D}
D.~Dhar.
\newblock Self-organized critical state of sandpile automaton models.
\newblock {\em Phys. Rev. Lett.}, 64(14):1613--1616, 1990.

\bibitem[HM82]{HM82}
J.~Harris and D.~Mumford.
\newblock On the {K}odaira dimension of the moduli space of curves.
\newblock {\em Invent. Math.}, 67(1):23--88, 1982.
\newblock With an appendix by William Fulton.

\bibitem[Len17]{L}
Y.~Len.
\newblock Hyperelliptic graphs and metrized complexes.
\newblock {\em Forum Math. Sigma}, 5:Paper No. e20, 15, 2017.

\bibitem[MZ25]{MZ25}
M.~Melo and A.~Zheng.
\newblock Tropical trigonal curves.
\newblock {\em arXiv:2501.03903}, 2025.

\end{thebibliography}
\end{document}